\documentclass[12pt]{amsart}

\usepackage{mathrsfs,amssymb}

\usepackage[all]{xy}

\makeatletter
\@namedef{subjclassname@2020}{%
	\textup{2020} Mathematics Subject Classification}
\makeatother

\usepackage[T1]{fontenc}
\newtheorem{theorem}{Theorem}[section]
\newtheorem{claim}[theorem]{Claim}

\newtheorem{lemma}[theorem]{Lemma}
\newtheorem{proposition}[theorem]{Proposition}
\newtheorem{conclusion}[theorem]{Conclusion}
\newtheorem{observation}[theorem]{Observation}
\newtheorem{corollary}[theorem]{Corollary}

\theoremstyle{definition}
\newtheorem{definition}[theorem]{Definition}

\newtheorem{problem}[theorem]{Problem}
\newtheorem{conjecture}[theorem]{Conjecture}
\newtheorem{discussion}[theorem]{Discussion}
\newtheorem{fact}[theorem]{Fact}

\newtheorem{hypothesis}[theorem]{Hypothesis}

\theoremstyle{remark}
\newtheorem{remark}[theorem]{Remark}
\newtheorem{question}[theorem]{Question}
\newtheorem{notation}[theorem]{Notation}

\newcommand{\ucp}{{\rm ucp}}

\newcommand{\iso}{{\rm iso}}
\newcommand{\sort}{{\rm sort}}

\newcommand{\supp}{{\rm supp}}

\newcommand{\ob}{{\rm Ob}}

\newcommand{\Hom}{{\rm Hom}}

\newcommand{\Ord}{{\rm Ord}}
\newcommand{\Aut}{{\rm Aut}}
\newcommand{\aut}{{\rm Aut}}

\newcommand{\id}{{\rm id}}

\newcommand{\bfd}{{\mathbf d}}

\newcommand{\bfc}{{\mathbf c}}

\newcommand{\bfp}{{\mathbf p}}
\newcommand{\bfx}{{\mathbf x}}

\newcommand{\bfH}{{\mathbf H}}

\newcommand{\bfy}{{\mathbf y}}

\newcommand{\dom}{{\rm dom}}

\newcommand{\Rang}{{\rm Rang}}
\newcommand{\Range}{{\rm Range}}

\newcommand{\rest}{{\restriction}}

\newcommand{\Wilog}{{\rm Without loss of generality}}

\newcommand{\mn}{{\medskip\noindent}}
\newcommand{\sn}{{\smallskip\noindent}}

\newcommand{\gB}{{\mathfrak B}}
\newcommand{\gA}{{\mathfrak A}}

\newcommand{\gC}{{\mathfrak C}}

\newcommand{\cH}{{\mathscr H}}

\newcommand{\bbZ}{{\mathbb Z}}

\newcount\skewfactor
\def\mathunderaccent#1#2 {\let\theaccent#1\skewfactor#2
\mathpalette\putaccentunder}
\def\putaccentunder#1#2{\oalign{$#1#2$\crcr\hidewidth
\vbox to.2ex{\hbox{$#1\skew\skewfactor\theaccent{}$}\vss}\hidewidth}}
\def\name{\mathunderaccent\tilde-3 }

\numberwithin{equation}{section}

\begin{document}

%%%%% To ease editing, for IMPAN journals add:

\baselineskip=17pt

%%%%%%%%%%%%%%%%

\title {Naturality and Definability III}
\author[M. Asgharzadeh]{Mohsen Asgharzadeh}
\address{Mohsen Asgharzadeh, Hakimiyeh, Tehran, Iran.}
\email{mohsenasgharzadeh@gmail.com}

\author[M.  Golshani]{Mohammad Golshani}
\address{Mohammad Golshani, School of Mathematics, Institute for Research in Fundamental Sciences (IPM), P.O.\ Box:
	19395--5746, Tehran, Iran.}
\email{golshani.m@gmail.com}

\author[S. Shelah] {Saharon Shelah}
\address{Einstein Institute of Mathematics\\
Edmond J. Safra Campus, Givat Ram\\
The Hebrew University of Jerusalem\\
Jerusalem, 91904, Israel\\
 and \\
 Department of Mathematics\\
 Hill Center - Busch Campus \\
 Rutgers, The State University of New Jersey \\
 110 Frelinghuysen Road \\
 Piscataway, NJ 08854-8019 USA}
\email{shelah@math.huji.ac.il}
\urladdr{http://shelah.logic.at}

\begin{abstract}
In this paper, we explore the relationship between the notions of naturality from category theory and definability from model theory. We study their interactions and present three main results. First, we show that under some mild conditions, naturality implies definability. Second, using reverse Easton iteration of Cohen forcing notions, we construct a transitive model of ZFC in which every uniformisable construction is weakly natural. Finally, we demonstrate that if $F$ is a natural construction on a class $\mathcal{K}$ of structures, represented by some formula, then it is uniformly definable without the need for extra parameters. Our results resolve some questions posed by Hodges and Shelah.
\end{abstract}


\maketitle

\section {Introduction}

We aim to explore the interplay between the notions of naturality from category theory and definability from  model theory. This continues the investigation initiated by Hodges and Shelah in \cite{Sh:160} and \cite{Sh:301}.

In this paper, naturality is understood in the sense of Eilenberg and Mac Lane \cite{EIL}. Specifically, for a class $\mathcal{K}$ of structures, the construction $A \mapsto F(A)$, where $A \in \mathcal{K}$, is natural if for every $A \in \mathcal{K}$ and every automorphism $a$ of $A$, there exists an automorphism $f(a)$ of $F(A)$ such that the following diagram commutes:
$$\xymatrix{
	&& A\ar[d]_{a}\ar[r]^{\subseteq}&F(A) \ar[d]^{f(a)}&\\
	&&  A \ar[r]^{\subseteq}& F(A)
	&&&}$$
and the assignment $A\mapsto F(A)$ has the following two properties:
\begin{enumerate}
\item  $f(\id_A)=\id_{F(A)}$, where $\id_A$ is the identity function on $A$,
\item for automorphisms $a, b$ of $A$,   $f(ab)=f(a)f(b)$.
\end{enumerate}

We remark that Hodges and Shelah \cite{Sh:301,Sh:160} have other requirements, for example,
that there is an embedding of A into $F(A)$ and that $F(A)$ is determined
up to isomorphism.

The initial challenging example for naturality was the dual of vector spaces, as discussed in \cite{EIL}. Hodges and Shelah \cite[Example 1]{Sh:160} observed that the constructions of algebraic closure of fields and of divisible hulls are not natural.
Recall that a  minimal divisible abelian group  containing an abelian group $G$ is called a divisible hull  of $G$. According to \cite[Theorem IV. 2.7]{fuchs}, such a divisible hull exists and is unique, and is denoted as $E(G)$. The historical origin of the question about the naturality of $E(-)$ comes from \cite[Page 10]{lamb}, where Lambek observed that the divisible hull $E(G)$ of an abelian group $G$ is not natural.

Given a category $\mathcal{C}$, let $\ob(\mathcal{C})$ denote the class of objects of $\mathcal{C}$, and for $X, Y \in \ob(\mathcal{C})$, let $\Hom_{\mathcal{C}}(X,Y)$ denote the class of morphisms from $X$ to $Y$ in $\mathcal{C}$. Now, let $G: \mathcal{C} \to \mathcal{D}$ be a functor, and recall that a functor $F: \mathcal{D} \to \mathcal{C}$ is  left adjoint  to $G$ if there is a canonical isomorphism
\[
\Hom_{\mathcal{C}}(F(X), Y) \cong \Hom_{\mathcal{D}}(X, G(Y)),
\]
where $X \in \ob(\mathcal{D})$ and $Y \in \ob(\mathcal{C})$. For more details, see Mac Lane’s book \cite[Page 81]{ma}. Suppose
 $G$ is the forgetful functor on structures and that $F$ is left adjoint to $G$.
 Recall from \cite{ma} that 
$F$ defines a natural construction, which can serve as a source for generating several other natural constructions.

Let $\mathcal{T}$ be the category of torsion-free abelian groups. For example, Lambek \cite[Page 11]{lamb} pointed out that the construction $G \mapsto E(G)$ from $\mathcal{T}$ is natural. Indeed, the divisible hull of $G \in \ob(\mathcal{T})$ can also be obtained as $\mathbb{Q} \otimes_\mathbb{Z} G$, as seen in \cite[Example IV. 2.9]{fuchs}. Hence, $E: \mathcal{T} \to \mathcal{T}$ is just the tensor product, which is left adjoint to the hom-functor, as explained in \cite[Page 80]{ma}. Therefore, the result \cite[Example 3]{Sh:160} confirms that $E: \mathcal{T} \to \mathcal{T}$ is a natural construction.
This line of research is further explored in \cite{adm}.

Let \( H \) and \( G \) be two (not necessarily abelian) groups. Suppose  
\( \varphi: H \twoheadrightarrow G \) is a surjective group homomorphism.  
Following \cite[Definition 3.1(i)]{Sh:301}, we say that a map  
\( \psi \in \Hom(G, H) \) splits  \( \varphi \) if  
$
\varphi \circ \psi = \id_G.
$ By \cite[Lemma 1]{Sh:160}, the naturality condition provides a splitting for  
the map \( \Aut(F(A)) \to \Aut(A) \). However, if \( \psi \) does not satisfy  
the group homomorphism property,
then Hodges and Shelah \cite[Definition 3.1(ii)]{Sh:301} introduced the notion  
of a map \( \psi: G \to H \)  ``weakly splitting''  \( \varphi \), provided it satisfies:  

\begin{center}  
	\( \varphi \circ \psi = \id_G \),  
	\( \psi(x^{-1}) = (\psi(x))^{-1} \),  
	\( \psi(e_G) = e_H \).  
\end{center}
Note that in this case, \( \psi \) is not necessarily a group homomorphism.  
Instead, we assume that the induced map belongs to \( \Hom(G, H / \mathcal{Z}(H)) \),  
where \( \mathcal{Z}(H) \) denotes the center of \( H \).  
We say that \( \varphi \) has \textit{lifting} (resp. \textit{weakly lifting})  
if there exists a map \( \psi \) that splits (resp. weakly splits) \( \varphi \).  
This leads to the notion of  weak naturality, as detailed in  
\cite[Definition 3.1]{Sh:301}.
Hodges and Shelah \cite{Sh:160} raised the following question:
\begin{question}\label{Q1}
When does naturality imply definablity?
\end{question}
We employ two-sorted models to address the question at hand. This approach is not new. For instance, Friedman \cite{fri} has already considered constructions in algebra (such as the direct product construction of pairs of groups) as operations from relational structures of a fixed, finite many-sorted relational type to structures of an enlarged many-sorted relational type. Similarly, Hodges and Shelah \cite{Sh:301} framed Question \ref{Q1} in terms of two-sorted models.
Thus, Question \ref{Q1} can naturally be approached within the context of many-sorted model theory.

In \S 2, for the convenience of the readers, we briefly review the concept of many-sorted models. According to its definition, any 2-sorted model $\gB$ provides two groups:
$$
H := \aut(\gB) \quad \text{and} \quad G := \aut(\sort_1(\gB)),
$$
as defined in Definition \ref{2sort}. If the restriction map $\varphi: H \to G$ is well-defined and has a weak lifting $\psi$, we say that $\gB$ gives us a \textit{uni-construction problem}:
$$
\bfc = \langle \gB, \sort_1(\gB), H, G, \varphi, \psi \rangle,
$$
as described in Definition \ref{y5}. The concept of a uni-construction problem will serve as the foundation for weak naturality.
Next, we introduce the concept of a \textit{\(\chi\)-solution} (see Definition \ref{y11}),  
which serves as the formal counterpart to definability.
  
Regarding Question \ref{Q1}, in \S 3, we discuss certain special types of three-sorted models  
and utilize Higman-Neumann-Neumann extensions to establish the following as our first main result:

\begin{theorem}
	\label{a23}
	Let $\chi$ be a cardinal and let $\bfc$ be a uni-construction problem. If $\bfc$   has no lifting, then   in a forcing extension, $\bfc$ has no
	$\chi$-solution.
\end{theorem}

The next auxiliary concept is uniformity, as defined in Definition \ref{yu11}. Specifically, a construction is uniformisable if its set-theoretic definition can be expressed in a way such that for each $A$, the corresponding $B$ is uniquely determined. The precise and original definition can be found in \cite[Definition 2.3]{Sh:301}.

In \cite[Theorem 5.1]{Sh:301}, Hodges and Shelah constructed a transitive model of ZFC in which every uniformisable construction of a prescribed size is weakly natural. This result highlights the advantage of weak naturality, as \cite[Theorem 4.3]{Sh:301} establishes that there is no transitive model of ZFC in which the natural constructions are precisely the uniformisable ones.
Since their result in \cite[Theorem 5.1]{Sh:301} includes a cardinality restriction, Hodges and Shelah raised the following natural conjecture:

\begin{conjecture}
\label{conjecture}(See \cite[Page 16]{Sh:301}).
The cardinality restriction can be removed.
\end{conjecture}

In \S 4, we settle the conjecture by iterating the main forcing construction from \cite[Theorem 5.1]{Sh:301}, using reverse Easton iteration of suitable Cohen forcing notions. It is worth noting that the use of forcing techniques in naturality problems dates back to Harvey Friedman \cite{fri}, where he applied Easton product Cohen-forcing. However, the concept of naturality in \cite{fri} is weaker than the one used here.

Let \( M \) be a model of set theory. Suppose that some formula within \( M \)  
represents the construction \( F \) on the class \( \mathcal{K} \in M \). If $F$ is natural and there are only a set number of isomorphism types of structures in $\mathcal{K}$, then \cite[Theorem 3]{Sh:160} asserts that $F$ is definable in $M$ with parameters. In this context, Hodges and Shelah raised the following problem:
\begin{problem}\label{p4} (See \cite[Problems (A) and (B)]{Sh:160}).
	\begin{enumerate}
	\item[(i)] Can the restriction ``$\mathcal K$ contains only a set of isomorphism types''
be removed?	\item[(ii)]  If $F$ has a  representing formula $\varphi$, is it always possible to define $F$ by a formula  whose parameters are those in $\varphi$
and those needed to define  $\mathcal K $?
\end{enumerate}
\end{problem}

%  Recall that left adjoint   defines a natural construction.

%We want to restrict parameters from the defining formula of the natural constructions.
In \S 5, we apply techniques from many-sorted models and reformulate the result of Hodges-Shelah \cite[Theorem 3]{Sh:160} about the implication that naturality implies definability.
To this end, suppose $\tau$ is a vocabulary, and $\mathcal{K}$ is a class of $\tau$-models. Suppose $\gB_1, \gB_2 \in \mathcal{K}$ and
$$ f_{\gA_1, \gA_2}: \gA_1 := \sort_1(\gB_1) \stackrel{\cong} \longrightarrow \sort_1(\gB_2) =: \gA_2. $$
Assume we have some construction $F: \gA_i \mapsto \gB_i$. We say $F$ is \textbf{uniform} if its definition is independent of the choice of $\gA_i$, and the extended isomorphism $F(f_{\gA_1, \gA_2})$ satisfies the following commutative diagram:
$$\xymatrix{
	&& \gA_1\ar[d]_{f_{\gA_1,\gA_2}}\ar[r]^{\subseteq}&F(\gA_1) \ar[d]^{F(f_{\gA_1,\gA_2})}&\\
	&&  \gA_2 \ar[r]^{\subseteq}& F(\gA_2)
	&&&}$$ is
independent of $f_{\gA_1,\gA_2}$.
%The idea
%of \cite[Theorem 3]{Sh:160} is to take ``average''.

%This reformulation helps us to control parameters, and enables us to
We prove    the following solution to  Problem \ref{p4}:

\begin{theorem}
	\label{d2}(Uniformity).
Let $\tau$ be a
vocabulary, and let
$\mathcal K$ be a class of $\tau$-models which is first order definable
from a parameter $\bfp$  such that    every $\gB \in \mathcal K$ is two sorted and the natural homomorphism $\varphi:\aut(\gB)\twoheadrightarrow\aut(\sort_1(\gB))$ splits.
 Then
	there exists a  class
	function  $$F:\{\sort_1(\gB):\gB \in \mathcal K\}\longrightarrow \mathcal K,$$  which is uniformly definable from
	the parameter $\bfp$ and  $\sort_1(F(\gA)) = \gA$, where $\gA = \sort_1(\gB)$.

	\iffalse
	If (A) then (B) where:
	\mn
	\begin{enumerate}
		\item[(A)]
		\begin{enumerate}
			\item[(a)]  $\tau$ is a vocabulary with 2-sorted so for a $\tau$-model
			$\gB,\gA = \sort_1(\gB)$ is a $\tau_1$-model, for suitable $\tau_1$
			\sn
			\item[(b)]  $K$ is a class of $\tau$-models so first order definable
			from the parameter $\bfp$
			\sn
			\item[(c)]  if $\gB \in K$ then the natural homomorphism $\pi^*_{\gB}$
			from $\aut(\gB)$ into $\aut(\sort_1(\gB))$ onto $\aut(\sort_1(\gB))$ does splits
		\end{enumerate}
		\sn
		\item[(B)]  using the parameter $\bfp$ uniformly we can define a class
		function $F$ such that:
		\sn
		\begin{enumerate}
			\item[(a)]  the domain is $K_1 = \{\sort_1(\gB):\gB \in K\}$
			\sn
			\item[(b)]  if $\gA \in K_1$ then $F(\gA) \in K$ and $\sort_1(F(\gA)) = \gA$.
		\end{enumerate}
	\end{enumerate}
	\fi
\end{theorem}
\iffalse
Let us talk about the uniformly definable function.
Suppose $\gB_1,\gB_2 \in\mathcal K$ are so that $$f_{\gA_1,\gA_2}:\gA_1:=\sort_1(\gB_1)\stackrel{\cong} \longrightarrow\sort_1(\gB_2)=:\gA_2.$$
Assume we have some construction $F:\gA_i\mapsto\gB_i$. By the uniformity of $F$, we mean its definition be independent of the choose  $\gA_i$, and the extended isomorphism
$F(f_{\gA_1,\gA_2})$ in the following
commutative  diagram:
$$\xymatrix{
	&& \gA_1\ar[d]_{f_{\gA_1,\gA_2}}\ar[r]^{\subseteq}&F(\gA_1) \ar[d]^{F(f_{\gA_1,\gA_2})}&\\
	&&  \gA_2 \ar[r]^{\subseteq}& F(\gA_2)
	&&&}$$ is
independent of $f_{\gA_1,\gA_2}$. The idea
of \cite[Theorem 3]{Sh:160} is to take ``average''Let $G$ and $H$ be two groups, which are not necessarily abelian.
Assume $\varphi: H \twoheadrightarrow G$ is a surjective homomorphism of groups.

\noindent
\begin{enumerate}
	\item[1)] We say that $\psi \in \Hom(G,H)$ splits $\varphi$ when $\varphi \circ \psi= \id_G$.
	
	\item[2)] We say that $\psi$ weakly splits $\varphi$ provided:
	\begin{enumerate}
		\item[(a)] $\psi$ is a function from $G$ into $H$,
		\item[(b)] $\varphi \circ \psi = \id_G$,
		\item[(c)] the composite mapping $\pi_H \circ \psi$ belongs to $\Hom(G,H/\mathcal{Z}(H))$,
		\item[(d)] $\psi(x^{-1}) = (\psi(x))^{-1}$ and $\psi(e_G) = e_H$.
	\end{enumerate}
	
	\item[3)] We say that $\varphi$ has a lifting (resp. weak lifting) if some $\psi$ splits (resp. weakly splits) it.
\end{enumerate} .
\fi
In the final section, we begin by showing that the proof of Theorem \ref{d2}  
can be reduced to the case where \( \mathcal{K} \) contains only one equivalence class.  
Let \( \gA \) be an arbitrary model.  
In general, there may be many choices for \( \gB \) and, consequently, many functions \( F \)  
such that \( F(\gA) = \gB \) and \( \sort_1(\gB) = \gA \). The parameters involved in \( \gB \)  
and \( F \) are typically wild, meaning they may lie outside the parameter set \( \bfp \).  
The main contribution of the uniformity theorem is that it guarantees the existence  
of a distinguished two-sorted model, which we call the \textit{universal model}. For more details, see Claim \ref{cla6}.  
The universal model is defined entirely in terms of the parameters in \( \bfp \).  
To construct this model, we rely on the concept of averages, as introduced in  
\cite[Theorem 3]{Sh:160}. Specifically, we use this approach to define a function \( F \)  
such that \( \sort_1(F(\gA)) = \gA \), with all parameters involved in the construction  
of \( \gB \) and \( F \) drawn from \( \bfp \).  
Thus, Theorem \ref{d2} can be viewed as a generalization of \cite[Theorem 3]{Sh:160}.

For all unexplained definitions from model theory  and category theory, see
the books of Hodges
 \cite{hodges}  and Mac Lane \cite{ma}, respectively. Also, for  unexplained definitions from the theory of forcing
see the book of Jech \cite{jech}.

\iffalse

The paper is organized as follows. In Section \ref{0B}, we review the notion of many sorted models, define the concepts of
uni-construction problem, lifting, weak lifting and solvability and prove some facts about them.
In Section \ref{fromuni}, we prove some results about uni-construction problems, in particular, we show that if a uni-construction problem
has no lifting, then in a forcing extension, it has no solution (up to some parameters).
 In Section
\ref{global-2}, we give a forcing construction, which confirms Conjecture \ref{conjecture}. Finally in Section
\ref{3}, we prove the uniformity Theorem \ref{d2}.
\fi

\section {The uni-construction problem} 
In this section we define the concepts of uni-construction problem, lifting, weak lifting, and solvability. We use many-sorted models to unify and simplify our treatment.
Let us start by defining many-sorted model structures. These structures are a suitable vehicle for dealing with statements concerning different types of objects, which are ubiquitous in mathematics. For our purpose, we only consider $n$-sorted model structures, where $n \geq 1$ is a natural number.
%We only work with finite structures.
\begin{definition}
	\label{2sort}
An $n$-sorted model structure $\mathcal{M}$ is of the form
$$\mathcal{M}= \big(\{M_1; \ldots ;M_n\};R_1; \ldots ;R_m; f_1; \ldots ; f_k ; c_1; \ldots ; c_l \big),$$
where
\begin{enumerate}
\item[(a)] the universes  $M_1; \ldots ;M_n$ are nonempty,

\item[(b)] the relations $R_1; \ldots ;R_m$ are between elements of the universes. In other words, for each $i$, there is
some sequence $s(R_i)=\langle i_1,\ldots,i_s \rangle$ from $\{1,\ldots n\}$ such that
$$R_i\subseteq M_{i_1}\times \ldots  \times M_{i_s},$$
\item[(c)] the functions $f_1; \ldots ; f_k$ are between elements of the universes. In other words, for each $i$,
  there is some sequence $s(f_i)=\langle i_1,\ldots,i_s,r \rangle$ from $\{1,\ldots n\}$ such that
  $$f_i: M_{i_1}\times \ldots  \times M_{i_s}\to M_r,$$
\item[(d)] the distinguished constants $\{c_1; \ldots ; c_l\}$ are in the universes, i.e., for each $i$ there is some
$s(c_i)=j \in \{1, \cdots, l\}$ such that $c_i \in M_j$.
\end{enumerate}
\end{definition}
\begin{remark}
An  ordinary single-sorted first order structure
$$\mathcal{M}=  \big( M; R_1; \ldots ;R_m; f_1; \ldots ; f_k ; c_1; \ldots ; c_l \big)$$
can be identified with the $1$-sorted model structure
$$\big( \{M \}; R_1; \ldots ;R_m; f_1; \ldots ; f_k ; c_1; \ldots ; c_l \big).$$
\end{remark}
\begin{notation}
\label{nota}
Given an $n$-sorted model structure $\mathcal{M}$ as above, for any non-empty set $I \subseteq \{1, \cdots, n \}$,
we can define the $|I|$-sorted model structure $\sort_{I}(\mathcal{M})$ as
 $$\sort_{I}(\mathcal{M})= \big(\{M_i: i \in I\};R_{i_1}; \ldots ;R_{i_t}; f_{j_1}; \ldots ; f_{j_p} ; c_{k_1}; \ldots ; c_{k_e} \big)$$
where
\begin{enumerate}

\item[(a)] $\{R_{i_1}; \ldots ;R_{i_t}\}$ is a subset of $\{ R_1; \ldots ;R_m  \}$,
consisting of those $R_i$ such that $i_1, \cdots, i_s \in I$, where $s(R_i)=\langle i_1,\ldots,i_s \rangle,$

\item[(b)] $\{f_{j_1}; \ldots ;f_{j_p}\}$ is a subset of $\{ f_1; \ldots ;f_k  \}$,
consisting of those $f_i$ such that $i_1, \cdots, i_s, r \in I$, where $s(f_i)=\langle i_1,\ldots,i_s, r \rangle,$

\item[(c)] $\{c_{k_1}; \ldots ; c_{k_e}\}$ is a subset of $\{ c_1; \ldots ;c_l  \}$,
consisting of those $c_i$ such that $j \in I$, where $s(c_i)=j.$
\end{enumerate}
In the case $I=\{i\}$, we denote the resulting structure by $\sort_i(\mathcal{M})$, and if $I=\{i, j   \}$,
where $1 \leq i <j\leq n,$ then we denote the resulting structure by $\sort_{i,j}(\mathcal{M})$.
%there is a  canonical model structure with the universe $M_i$, that we denote by $\sort_i(\mathcal{M})$. Similarly, for any $1 \leq i <j\leq n,$ there is a  canonical model structure with the universe $M_i\times M_j$, that we denote by $\sort_{i,j}(\mathcal{M})$.
\end{notation}
The notions of homomorphism and isomorphism between
n-sorted models are naturally defined. We present them here for completeness.
\begin{definition}
\label{def1}
suppose
$\mathcal{M}= \big(\{M_1; \ldots ;M_n\};R_1; \ldots ;R_m; f_1; \ldots ; f_k ; c_1; \ldots ; c_l \big)$
and
$\mathcal{M}'= \big(\{M'_1; \ldots ;M'_n\};R'_1; \ldots ;R'_m; f'_1; \ldots ; f'_k ; c'_1; \ldots ; c'_l \big)$
are two $n$-sorted model structures of the same sorts.
\begin{enumerate}
\item A homomorphism $\pi$ from $\mathcal{M}$ to $ \mathcal{M}'$, denoted $\pi: \mathcal{M} \rightarrow \mathcal{M}',$
is a function $\pi: \bigcup_{i=1}^n M_i \rightarrow \bigcup_{i=1}^n M'_i$ such that:
\begin{enumerate}
\item For each $1 \leq i \leq n, \pi \upharpoonright M_i: M_i \rightarrow M'_i$,

\item For each $1 \leq i \leq m$, if $s(R_i)=\langle i_1,\ldots,i_s \rangle$ and $(a_1, \cdots, a_s) \in M_{i_1} \times \cdots \times M_{i_s},$
then
\[
R_i(a_1, \cdots, a_s) \Leftrightarrow R'_i(\pi(a_1), \cdots, \pi(a_s)),
\]
\item For each $1 \leq i \leq k$, if $s(f_i)=\langle i_1,\ldots,i_s, r \rangle$ and $(a_1, \cdots, a_s) \in M_{i_1} \times \cdots \times M_{i_s},$
then
\[
\pi(f_i(a_1, \cdots, a_s))=f'_i(\pi(a_1), \cdots, \pi(a_s)),
\]
\item For each $1 \leq i \leq l,$ $\pi(c_i)=c'_i.$
\end{enumerate}
\item A homomorphism $\pi: \mathcal{M} \rightarrow \mathcal{M}'$ is an isomorphism, if for each $1 \leq i \leq n, \pi \upharpoonright M_i: M_i \rightarrow M'_i$
is a bijection.

\item An automorphism of $\mathcal{M}$ is an isomorphism $\pi: \mathcal{M} \rightarrow \mathcal{M},$ and denote this by $\pi\in\aut(\mathcal{M})$.
\end{enumerate}
\end{definition}
In what follows, we focus solely on 2-sorted and 3-sorted model structures.

	\begin{discussion}
For a group $L$, by $e_L$ we mean the unit element. We denote the group operation by $\cdot$, so that for two elements $\ell_1, \ell_2 \in L$, their product is denoted by $\ell_1 \cdot \ell_2$. By $\ell^{-1}$ we mean the inverse of $\ell \in L$. As usual, if $L$ is abelian, we use the additive notation $(L, +, -, 0)$.
\end{discussion}
\begin{notation}Let \( H \) be a group, not necessarily abelian. We define the following:
	
	\begin{enumerate}
		\item[i)] Let $\mathcal{Z}(H)$ denote the center of $H$, i.e.,
		$$\mathcal{Z}(H) := \{ h \in H : hx = xh \quad \forall x \in H \}.$$
		
		\item[ii)] Since $\mathcal{Z}(H)$ is a normal subgroup of $H$, it induces a group structure on the quotient $H/\mathcal{Z}(H)$. The canonical group homomorphism from $H$ onto $H/\mathcal{Z}(H)$ is given by $\pi_H: H \to H/\mathcal{Z}(H)$, defined by $h \mapsto h \cdot \mathcal{Z}(H)$.
		
		\item[iii)] We denote the set of all group homomorphisms from a given group $G$ to $H$ by $\Hom(G, H)$. A function $f: G \to H$ belongs to $\Hom(G, H)$ if it satisfies $f(g_1g_2) = f(g_1)f(g_2)$ for all $g_1, g_2 \in G$.
		
		\item[iv)] By $\id_G \in \Hom(G, G)$, we mean the identity map on $G$.
	\end{enumerate}
\end{notation}

	\begin{definition}
\label{y2} Let $G$ and $H$ be two groups, not necessarily abelian.
Assume $\varphi: H \twoheadrightarrow G$ is a surjective homomorphism of groups.

\noindent
\begin{enumerate}
	\item[1)] We say that $\psi \in \Hom(G,H)$ splits $\varphi$ when $\varphi \circ \psi= \id_G$.
	
	\item[2)] We say that $\psi$ weakly splits $\varphi$ provided:
	\begin{enumerate}
		\item[(a)] $\psi$ is a function from $G$ into $H$,
		\item[(b)] $\varphi \circ \psi = \id_G$,
		\item[(c)] the composite mapping $\pi_H \circ \psi$ belongs to $\Hom(G,H/\mathcal{Z}(H))$,
		\item[(d)] $\psi(x^{-1}) = (\psi(x))^{-1}$ and $\psi(e_G) = e_H$.
	\end{enumerate}
	
	\item[3)] We say that $\varphi$ has a lifting (resp. weak lifting) if some $\psi$ splits (resp. weakly splits) it.
\end{enumerate}
\end{definition}
We introduce the concept of the uni-construction problem, which plays a key role in this paper.
\begin{definition}
\label{y5}
1) We say $\bfc$ is a uni-construction problem ($\ucp$ in short),
when
\[
\bfc=\langle \gB_{\bfc}, \gA_{\bfc}, H_{\bfc}, G_{\bfc}, \varphi_{\bfc}, \psi_{\bfc}         \rangle =
\langle \gB, \gA, H,  G, \varphi, \psi         \rangle
\]
and it satisfies the following
conditions:
\begin{enumerate}
\item[(a)]  $\gB$ is a two-sorted model,

\item[(b)]  $\gA:=\sort_1(\gB),$

\item[(c)]  $H =\aut(\gB)$ and $G = \aut(\gA)$,

\item[(d)]  $\varphi$ is the natural restriction map from $H$ into $G$,
  i.e. $\varphi(f) = f \rest \gA$:
   	$$\xymatrix{
  	&& \gA\ar[d]_{\subseteq}\ar[r]^{f\rest \gA}&\gA \ar[d]^{\subseteq}&\\
  	&& \gB\ar[r]^{{f}}& \gB
  	&&&}$$
Recall that $\varphi$ is a group homomorphism. Indeed,
given an automorphism $f, \varphi(f)$ is its restriction to
the first sort, therefore it is an automorphism; furthermore, restriction preserves the
composition.
Let us depict the resulting   commutative diagram:
$$\xymatrix{
	&& \gA\ar[dd]_{\subseteq}\ar[r]^{f\rest \gA}&\gA \ar[dd]^{\subseteq}&\\
	& \gA\ar[dd]_{\subseteq}\ar[rru]_{\quad(fg)\rest \gA}\ar[ru]^{g\rest \gA} &\\
	&& \gB\ar[r]^{{f}}& \gA &\\
	& \gB \ar[ru]^{{g}}\ar[rru]_{{fg}}
	&&&}$$

\item[(e)]  $\varphi$ is onto $G$. So, for any  $g:\gA\to \gA $ there is an $f$ such that the following diagram commutes	$$\xymatrix{
	&& \gA\ar[d]_{\subseteq}\ar[r]^{g}&\gA \ar[d]^{\subseteq}&\\
	&& \gB\ar[r]^{\exists{f}}& \gB
	&&&}$$

\item[(f)]  $\psi$ weakly splits $\varphi$.
\end{enumerate}

2) We say $\bfc$ is a weak uni-construction problem, if $\bfc$ is as above and it satisfies items (a)-(e) above.
\end{definition}

\begin{definition}
	\label{y51}
Let $\bfc$ be a (weak) uni-construction problem.
Then, the classes $\mathcal{K}^1_{\bfc}$ and $\mathcal{K}^2_{\bfc}$ are defined as follows:
\begin{enumerate}
	\item[(a)]  $\mathcal K^1_{\bfc}: = \{\gA:\gA$ isomorphic
	to $\gA_{\bfc}\}$,
	\sn
	\item[(b)]
 $\mathcal K^2_{\bfc} :=
\{\gB:\gB$ isomorphic to $\gB_{\bfc}\}$.
\end{enumerate}

%By $F$ we mean a class function $F:K^1_{\bfc}\to K^2_{\bfc}$.
\end{definition}

\begin{remark}
The assignment $\sort_1(\gB) \stackrel{F}{\mapsto} \gB$ on the domain of $\mathcal{K}^1_{\bfc}$ is not, in general,
single-valued; however, by clause (e) from Definition \ref{y5}, it is single-valued up to isomorphism over $\gB$.
This means that $\gB_1 \cong \gB_2$ provided that
$\sort_1(\gB_1) \cong \sort_1(\gB_2)$.

In particular, suppose that
$\gA \cong \sort_1(\gB)$. Then,
\[
F(\sort_1(\gB)) = \gB \cong F(\gA).
\]
\end{remark}

\begin{notation}\label{hchi}
	Let $\chi$ be   be infinite cardinal. By $\cH(\chi)$, we mean the collection of sets of hereditary cardinality
	less than $\chi$.
\end{notation}

\begin{definition}
\label{y11} Let $\bfc$ be a (weak) uni-construction problem.
\begin{enumerate}
	\item[(a)] We say $\bfc$ is solvable, when  there is a class function $F:\mathcal K^1_{\bfc}\to \mathcal K^2_{\bfc}$, such that for each
$\gB \in \mathcal K^2_{\bfc},$ $F(\sort_1(\gB))=\gB$.

	\item[(b)] Let $F$ be as in clause (a). We say $F$
is
definable, if there is a formula $\theta$ such that
$$F(\gA)=\gB \Leftrightarrow  \theta (\gA,\gB),$$ for all two sorted models $\gB$ with $\gA:=\sort_1(\gB)$.
	\item[(c)]
 We say $\bfc$ is purely solvable, when  there is an $F$ as above which is definable using only
$(\gB_{\bfc},\gA_{\bfc})$ as a parameter.

	\item[(d)] We say $\bfc$ is $\chi$-solvable, when there is $F$ as   above, which is
definable by some parameter $a \in \cH(\chi)$.
\end{enumerate}
\end{definition}

The notion of expansion and reduction between many sorted model structures is defined in the natural way. We give its definition for completeness.
\begin{definition}
\label{reduction}
Let $n_2,  n_1  \geq 1$, and suppose	
$$\gB_{1}=\big(\{M^1_1; M^1_2; \cdots; M^1_{n_1} \};R^1_1; \ldots ;R^1_{m_1}; f^1_1; \ldots ; f^1_{k_1} ; c^1_1; \ldots ; c^1_{l_1} \big)$$
and
$$\gB_{2}=\big(\{M^2_1; M^2_2; \cdots; M^2_{n_2} \};R^2_1; \ldots ;R^2_{m_2}; f^2_1; \ldots ; f^2_{k_2} ; c^2_1; \ldots ; c^2_{l_2} \big)$$
 are many-sorted models. We say $\gB_{2}$ expands $\gB_{1}$, or $\gB_{1}$ is a reduction of $\gB_{2}$,
if $\gB_{1}$ is obtained from $\gB_{2}$  by leaving out some sorts, relations, functions, and
constants. In other words,
\begin{enumerate}
\item[(a)] $n_2 \geq n_1$ and $\{M^1_1; M^1_2; \cdots; M^1_{n_1} \} \subseteq \{M^2_1; M^2_2; \cdots; M^2_{n_2} \},$

\item[(b)] For each $1  \leq i \leq m_1$, there is some $1 \leq j \leq m_2$ such that $R^1_i=R^2_j,$

\item[(c)] For each $1  \leq i \leq k_1$, there is some $1 \leq j \leq k_2$ such that $f^1_i=f^2_j,$

\item[(d)] For each $1  \leq i \leq l_1$, there is some $1 \leq j \leq l_2$ such that $c^1_i=c^2_j.$
\end{enumerate}
Note that relations, functions, and constants which are not meaningful are at
the same time dropped.
\end{definition}
\begin{remark}
If $\mathcal{M}$ is an $n$-sorted model structure as in Definition \ref{2sort} and $I \subseteq \{1, \cdots, n\}$
is non-empty, then $\sort_I(\mathcal{M})$, as defined in Notation \ref{nota}, is a reduction of $\mathcal{M}$.
\end{remark}

We conclude this section with the following simple observation.
\begin{observation}
	\label{a5}
	Let $\bfc$ and $\bfd$  be two uni-construction problems, satisfying the following conditions:
	\begin{enumerate}
		\item[(a)]    $\gA_{\bfc} = \gA_{\bfd}$,
		
		\item[(b)]  $\gB_{\bfd}$ expands $\gB_{\bfc}$.
	\end{enumerate}
	If $\bfd$ is solvable, then $\bfc$ is solvable.
\end{observation}
\begin{proof}
	Assume $\bfd$ is solvable, and let
	$F_{\bfd}:\mathcal K^1_{\bfd}\to \mathcal K^2_{\bfd}$ witness it. Let $\gA\in \mathcal K^1_{\bfc}$. Since $\mathcal K^1_{\bfd}=\mathcal K^1_{\bfc}$, we can define $F_{\bfd}(\gA)$.
	Let $F_{\bfc}(\gA)$ be  the reduction  of $F_{\bfd}(\gA)$ into the language of   $\gB_{\bfc}$. Then
	$F_{\bfc}:\mathcal K^1_{\bfc}\to\mathcal  K^2_{\bfc}$ is well-defined, and  it witnesses that $\bfc$ is solvable.
\end{proof}
\section {From uni-construction to naturality }

The main result of this section is Theorem \ref{a23}, which shows that for a given cardinal $\chi$, if a uni-construction problem has no lifting, then it has no $\chi$-solution in any forcing extension of the universe.
\begin{hypothesis}\label{hy}
Let  $\gC$ be a 3-sorted model.
\begin{enumerate}
	\item[(i)]

Let us fix the following 2-sorted models $\bfc_{1,2}$, $\bfc_{2,3}$ and $\bfc_{1,3}$ via defining their first and second sorts:
\begin{enumerate}
	\item[(1)]  $\gA_{\bfc_{1,2}} = \sort_1(\gC),\gB_{\bfc_{1,2}} =
	\sort_{1,2}(\gC)$,
	\sn
	\item[(2)]  $\gA_{\bfc_{2,3}} = \sort_{1,2}(\gC),\gB_{\bfc_{2,3}} = \gC$,
	\sn
	\item[(3)]  $\gA_{\bfc_{1,3}} = \sort_1(\gC),\gB_{\bfc_{1,3}} = \gC$.
\end{enumerate}

	\item[(ii)]
We assume in addition to $(i)$ that   $\bfc_{1,2},\bfc_{2,3}$ and $\bfc_{1,3}$ are weak
uni-construction problems. In particular, there are the following surjective group homomorphisms induced by the canonical restriction maps:

\begin{enumerate}
	\item[(1)]  $\varphi_{\bfc_{1,2}}:\aut(\sort_{1,2}(\gC))\to \aut(\sort_1(\gC))$,
	\sn
	\item[(2)] $\varphi_{\bfc_{2,3}}:\aut( \gC )\to \aut(\sort_{1,2}(\gC))$,
	\sn
	\item[(3)] $\varphi_{\bfc_{1,3}}:\aut( \gC )\to \aut(\sort_1(\gC))$.
\end{enumerate}

\end{enumerate}

\end{hypothesis}

\begin{lemma} (Transitivity)
\label{a8t}
Let  $\gC$ be as in Hypothesis \ref{hy}.
If $\bfc_{1,2}$  and  $\bfc_{2,3}$ are solvable, then $\bfc_{1,3}$ is solvable.
\end{lemma}

\begin{proof}
According to Definition \ref
{y11}
there are definable class functions
\begin{enumerate}
	\item[]  $F_{\bfc_{1,2}}:\mathcal K^1_{\bfc_{1,2}}\to \mathcal K^2_{\bfc_{1,2}}$, and
	\sn
	\item[]  $F_{\bfc_{2,3}}:\mathcal K^1_{\bfc_{2,3}}\to \mathcal K^2_{\bfc_{2,3}}$.
\end{enumerate}
 Recall from Hypothesis \ref{hy} that the following three equalities are satisfied
 \begin{enumerate}
 	\item[(i)]  $\mathcal K^1_{\bfc_{1,2}}=\sort_1(\gC)=\mathcal K^1_{\bfc_{1,3}} $,
 	\sn
 	\item[(ii)]  $ \mathcal K^1_{\bfc_{2,3}}=\sort_{1,2}(\gC)=\mathcal K^2_{\bfc_{1,2}}$,
 	\item[(iii)]   $ \mathcal K^2_{\bfc_{2,3}}=\gC=\mathcal K^2_{\bfc_{1,3}}$.
 \end{enumerate}

   This enables us to get the composition    $F_{\bfc_{1,3}}=F_{\bfc_{2,3}}\circ F_{\bfc_{1,2}}$.
Let us summarize things in the following   commutative diagram:

$$\xymatrix{
	&&\mathcal K^1_{\bfc_{1,3}}\ar[r]^{F_{\bfc_{1,3}}}\ar[dl]_{=}&\mathcal K^2_{\bfc_{1,3}}\\&  \mathcal K^1_{\bfc_{1,2}} \ar[dr]_{F_{\bfc_{1,2}}}&&&\mathcal K^2_{\bfc_{2,3}}\ar[ul]_{=}\\
	&&\mathcal K^2_{\bfc_{1,2}}\ar[r]^{=}&\mathcal K^1_{\bfc_{2,3}} \ar[ur]_{F_{\bfc_{2,3}}}\\
 &&&&}$$
Thus, the function $F_{\bfc_{1,3}}$ serves as evidence that $\bfc_{1,3}$ is solvable.
\end{proof}
Let us reformulate the transitivity in  terms  of $\chi$-solvablity:
\begin{lemma}
	\label{a8}
	Let  $\gC$ be as in Hypothesis \ref{hy}.
	If $\bfc_{1,2}$  and  $\bfc_{2,3}$ are $\chi$-solvable, then $\bfc_{1,3}$ is $\chi$-solvable.
\end{lemma}
\begin{conclusion}
\label{a11}

Let  $\gC$ be as in Hypothesis \ref{hy} and let $\chi$ be a cardinal.
If $\bfc_{1,3}$ is not $\chi$-solvable and $\bfc_{2,3}$ is solvable.
Then  $\bfc_{1,2}$ is not $\chi$-solvable.
\end{conclusion}

\begin{proof}
Suppose that  $\bfc_{1,2}$ is   $\chi$-solvable. It follows from Lemma
 \ref{a8} that $\bfc_{1,3}$ is  $\chi$-solvable, a contradiction.
\end{proof}

\begin{corollary}
\label{a14}
  Let $\gC$ be as in Hypothesis \ref{hy}, and let $\chi$ be an infinite cardinal. Assume, in addition, that
  \begin{enumerate}
  	\item[(a)] $\varphi_{\bfc_{1,3}}$ has no weak lifting, and
  	\item[(b)] $\varphi_{\bfc_{2,3}}$ has a lifting.
  \end{enumerate}

Then  there is a  forcing extension of the universe in which  $\bfc_{1,2}$ is not $\chi$-solvable.
\end{corollary}

\begin{proof}
Fix an infinite cardinal $\chi$. Since $\varphi_{\bfc_{2,3}}$ has a lifting, we observe, in light of \cite[Theorem 3]{Sh:160}, that it is solvable. Hence, for some cardinal $\chi' \geq \chi$, it is $\chi'$-solvable.

On the other hand, since $\varphi_{\bfc_{1,3}}$ has no weak lifting, by \cite[Theorem 4]{Sh:160}, we can find a generic extension $V[G]$ of the universe in which $\varphi_{\bfc_{1,3}}$ is not $\chi'$-solvable. In view of Conclusion \ref{a11}, we know that $\bfc_{1,2}$ is not $\chi'$-solvable, and therefore, it is not $\chi$-solvable in $V[G]$ either.
\iffalse
Recall from 	
\cite[Theorem 3]{Sh:160} that 	
	every model  every ucp is solvable provided it has a lifting.
	We combine this along with (b) to observe  consistently that $\bfc_{2,3}$ is solvable. Let $\chi$ be any cardinal.
Let also 	recall from
	\cite[Theorem 4]{Sh:301}
	that there
	is a model of ZFC in which any $\chi$-solvable ucp has a weak lifting.
	We apply this along with clause $(a)$, and after some forcing,
	to find a   model of ZFC such that $\bfc_{1,3}$ is not solvable.
	Thanks to
Corollary \ref{a11} now consistently $\bfc_{1,2}$ is not solvable.
\fi
\end{proof}
\begin{notation}
For a group $G$ and an automorphism $\psi$ of $G$,  let:
\begin{itemize}
\item $\psi^0:=\id_G$ denote the identity automorphism,

\item For $n>0, \psi^n$ is the composition of $\psi, n$-times,

\item For $n >0$, $\psi^{-n}$ is defined as $\psi^{-n}(x)= (\psi^n(x))^{-1}$.
\end{itemize}
\end{notation}
\begin{definition}\label{skew}
Let $G'$ be a group, and let $\psi$ be an automorphism of $G'$.
We are going to define a group-structure over the following set:
 $$G:=\{y^nx: n\in\mathbb{Z}\emph{ and }x\in G' \},$$
 via the following rules:
 	\begin{enumerate}
 	\item[(a)]  The identity element is $y^0e_{G'}$,
 	\sn
 	\item[(b)]    The multiplication $(y^nx_1)\times (y^mx_2)$  is defined by $y^{n+m}( \psi^m(x_1)\psi^n(x_2))$, and recall that $\psi^0$ is the identity map.
 		\item[(c)] The inverse of $(y^n x)$ is
$y^{-n}  x^{-1}  $, i.e.,$$(y^n x)\times (y^{-n}  x^{-1}) =y^{n-n}   \psi^{-n}(x) \psi^{n} (x^{-1})=y^0\psi^0(e_{G'})=y^0e_{G'}=e_G.$$
 \end{enumerate}
 We denote the resulting group by $G:=\mathbb{Z}\ltimes_{\psi}G'$.
\end{definition}

\begin{remark} Adopt the notation of Definition \ref{skew}.
	
	(1) The set $\{y^ne_{G'}\mid n\in\mathbb{Z}\}$ is a subgroup of $G$, and there is an isomorphism   $\pi:(\mathbb{Z},+)\to \{y^ne_{G'}\mid n\in\mathbb{Z}\}$ of
groups, where $\pi(n):=y^ne_{G'}$.

(2) The assignment
$x\mapsto y^0x$ defines a morphism $G'\to G$.
%Similarly, the assignment
%$y^n\mapsto y^ne_{G'}$ defines a morphism $\pi(\mathbb{Z})\to G$.

(3) For simplicity, let us denote $y^0x=x$ and $ye_{G'}=y$. This shows that
 $$x\times y^n:=y^0x\times y^ne_{G'}=
 y^{0+n}(\psi^n(x)\psi^0(e_{G'}))=
 y^n\psi^n(x).$$
Similarly, by using the above identification,
we have $y^{-n}xy^n=\psi^n(x).$
\end{remark}

The above construction is a special case of
HNN extensions, where HNN is an abbreviation for  Higman-Neumann-Neumann.
Here, is an algebraic definition equipped with a topological motivation:

\begin{discussion}(See \cite[\S 11: HNN extensions]{ro})
	
	(i) Suppose $G$ is a group, $A$ and $B$ are  two isomorphic subgroups of  $G$ and $f:A\stackrel{\cong}\longrightarrow B$. Then the group having the presentation
	$$(G; p\mid p^{-1}ap = f(a)\quad \forall a \in A)$$
	is called an HNN extension of $G$;
	
	(ii) Consider a
manifold $X$ with homeomorphic disjoint subspaces $A$ and
	$B$, and let $f: A \to B$ be a homeomorphism. Let $\hat{X}$ be $X$ that we add a handle from $A$ to $B$.
Then fundamental group of $\hat{X}$ is the HNN extension over the fundamental group of ${X}$ (see \cite[Theorem 11.75]{ro}).
\end{discussion}

\begin{notation}\label{sup}
	Suppose $\{H_i:i\in I\}$ is a family of groups.
	Let  $H$ be the family of all sequences $(h_i)_{i\in I}$ so that $h_i\in H_i$  with the property that $\{i:h_i\neq e_{H_{i}}\}$ is  finite, and we define $$\supp((h_i)_{i\in I}):=\{i:h_i\neq e_{H_{i}}\}.$$
	Now, take $(h_i),(g_i)\in H$ and define
	$(h_i)_{i\in I}\times(g_i)_{i\in I}:=(h_ig_i)_{i\in I}\in H$. By this multiplicative operation, $H$ equips with a  structure of  group. A common notation for $H$, is $\bigoplus_{i\in I}H_{i}$.
	
\end{notation}
\begin{proposition}
\label{a17}
Let $G_1$ and $G_2$ be two groups, and suppose there is an onto homomorphism $\varphi_{1,2}: G_2 \twoheadrightarrow G_1$ that has no lifting but has a weak lifting. Then, there exist a group $G_3$ and a homomorphism $\varphi_{2,3}$ such that:

\begin{enumerate}
\item[(a)]   $G_3$ is a group;

\item[(b)]    $\varphi_{2,3}:G_3\twoheadrightarrow G_2$ is a surjective  homomorphism
  which has a lifting;

\item[(c)]  $\varphi_{1,3} = \varphi_{1,2} \circ \varphi_{2,3} \in
  \Hom(G_3,G_1)$ is a surjective  homomorphism with no weak lifting.
\end{enumerate}
\end{proposition}

\begin{proof}(a):
First, we introduce the auxiliary group $G'_3 := \bigoplus\{G_{2,n}:n \in \bbZ\}$, where $G_{2,n} \cong G_2$ for each $n \in \bbZ$. Let also $\psi_n$ denote the corresponding isomorphism $\psi_n: G_2 \to G_{2,n}$. Note that:
\begin{enumerate}
	\item[(1)] $\bigcup_{n \in \bbZ} G_{2,n}$ generates $G'_3$,
	\item[(2)] $G_{2,n} \cap \bigcup_{k \in \bbZ \backslash \{n\}} G_{2,k} = \{e_{G'_3}\}$,
	\item[(3)] the family of groups $\{G_{2,n}: n \in \bbZ\}$ pairwise commutes.
\end{enumerate}

Let $\psi_*$ be the automorphism of $G'_3$ such that, for each $m$,
$$\psi_* \rest G_{2,m} = \psi_{m+1} \circ \psi_m^{-1}.$$
Set $G_3 := \mathbb{Z} \ltimes_{\psi_*} G'_3$. Recall from Definition \ref{skew} that $G_3$ is generated by $G'_3 \cup \{y\}$, subject to the following relations:
$$ y^{-1} x y = \psi_*(x) \quad \forall x \in G'_3.$$

 (b):
Let $\varphi_{2,3}:G_3 \rightarrow G_2$ be the unique homomorphism
from $G_3$ onto $G_2$ such that:

\begin{enumerate}
\item[(1)]  $\varphi_{2.3} \rest G_{2,n} =
\psi^{-1}_n \text{ for } n \in \bbZ$,

\item[(2)]  $\varphi_{2.3}(y) = e_{G_2}$.
\end{enumerate}

Now, $\psi_0$ is clearly a lifting of $\varphi_{2,3}$, i.e. clause (b)
holds.

$(c)$: In order to prove clause (c), we first prove the following claim.

\begin{claim}
\label{claim}
	 $G_3$ has trivial center.
\end{claim}
\begin{proof} Suppose $x \in G'_3 \backslash \{e_{G'_3}\}$. Clearly, we have $y^{-1} x y \ne x$, as can be seen using $\supp(x)$ (see Notation \ref{sup}). Now, let $\alpha \in G_3 \backslash G'_3$. By Definition \ref{skew}, $\alpha$ has the form $y^n x_1$, where $n \in \mathbb{Z}$, $x_1 \in G'_3$, and
	either $n \ne 0$ or $x_1 \ne e_{G'_3}$.
	
	Next, choose $x_2 \in G'_2 \backslash \{e_{G'_2}\}$ such that
	\[
	\ell := \min(\supp(x_2)) > \max(\supp(x_1)) + n.
	\]
	
	Now, assume for contradiction that $y^n x_1$ and $x_2$ commute. This leads to:

$$\begin{array}{ll}
y^n (x_1x_2)&=(y^n x_1)x_2\\
&=x_2(y^n x_1)\\
 &=x_2(yy^{n-1} x_1)\\
&= (x_2y)y^{n-1} x_1\\
&=(y\psi_*(x_2))y^{n-1} x_1\\
&=y^{n}(\psi_*^n(x_2) x_1).
\end{array}$$
By definition, this implies
$$ \psi_*^n(x_2) x_1 = x_1x_2 \quad (+). $$

Recall that $\psi_* \rest G_{2,m} = \psi_{m+1} \circ \psi^{-1}_m$. By combining this equality with $(+)$, we observe that at least one coordinate on the left-hand side is nonzero, whereas the corresponding coordinate on the right-hand side is also nonzero. This contradiction confirms that $G_3$ has a trivial center, as required.
\end{proof}
Now, we proceed with the proof of the proposition. By definition, $\varphi_{1,3} \in \Hom(G_3, G_1)$ is surjective. We aim to show that it has no weak lifting. Suppose, for the sake of contradiction, that there exists a function $\psi: G_1 \rightarrow G_3$ such that $\varphi_{1,3} \circ \psi = \id_{G_1}$, and that the composite map $\pi_3 \circ \psi$ belongs to $\Hom(G_1, G_3/\mathcal{Z}(G_3))$, where $\pi_3: G_3 \rightarrow G_3/\mathcal{Z}(G_3)$ is the canonical homomorphism.
However, according to Claim \ref{claim}, $G_3$ has a trivial center. Therefore, $\psi$ must be a homomorphism from $G_1$ into $G_3$. As a result, $\varphi_{2,3} \circ \psi$ is a homomorphism from $G_1$ into $G_2$, and in fact, it is injective, since $\varphi_{1,3} \circ \psi$ is injective. This contradicts the assumption that $\varphi_{1,2}$ has no weak lifting.
\end{proof}

\begin{corollary}
\label{a20}Let $\chi$ be a cardinal.
Suppose $\varphi \in \Hom(G_2,G_1)$ is onto, and does not split. Then there exists a weak uni-construction problem  $\bfc$,
such that  in some forcing extension,
 $\bfc$ is not $\chi$-solvable.
\end{corollary}

\begin{proof}
%Recall that  $\bfc$ gives us a data $(G_{\bfc},H_{\bfc},\varphi_{\bfc}) = (G_1,G_2,\varphi)$.
Let $G_1$ and $G_2$ be as above and set $\varphi_{1,2} := \varphi$. By Proposition \ref{a17}, we can find a group $G_3$ and two homomorphisms
$\varphi_{2,3}$ and $\varphi_{1,3}$, which fit in the following diagram
 $$\xymatrix{
	&  G_2\ar[rr]^{\varphi_{1,2}}&&G_1\\
	&& G_3\ar[ur]_{\varphi_{2,3} \rightsquigarrow\emph{ with no weak lifting}}\ar[ul]^{\emph{with lifting }\rightsquigarrow\varphi_{1,3}}
	&&&\\}$$
 \Wilog \, the groups $G_1, G_2$
 and $G_3$
are pairwise disjoint. We define the three sorted model $\gC$ as
follows:

\begin{enumerate}
\item[$(*)$]
\begin{enumerate}
\item[(a)]  the set of elements of $\sort_\ell(\gC)$ is the set of
  elements of $G_\ell$ for $\ell=1,2,3$;

\item[(b)]  $F^{\gC}_\ell = \varphi_{\ell,\ell +1}$ for $\ell=1,2$;

\item[(c)]  for $\ell=1,2,3$ and $a \in G_\ell$, the homomorphism $F^{\gC}_{\ell,a}:G_\ell \rightarrow G_\ell$ is
defined as  $F_{\ell,a}^{\gC}(b) = a \cdot b$ for $b \in G_\ell$.
\end{enumerate}
\end{enumerate}

Now, we show that the assumptions of Corollary \ref{a14} hold.
 To this end, note that we consider the sorts only as a set of elements and forget about
the multiplication in $G_i$. The entire structure is given by the functions.
We bring the following two claims:
\begin{enumerate}
	\item[(i)] $G_1 \cong \Aut(\sort_1(\gC)), G_2 \cong \Aut(\sort_{1,2}(\gC))$ and $G_3 \cong \Aut(\gC)$ and
	\item[(ii)] $\varphi_{i,j}:\Aut(\sort_{i,j}(\gC))\to  \Aut(\sort_i(\gC))$ is the natural restriction map.
\end{enumerate}
For clause (i), we only show that $G_3 \cong \aut(\gC)$, as the other cases
can be proved in a similar way.

In order to see $G_3 \cong \aut(\gC)$, define the  map
$$\theta: G_3 \rightarrow \aut(\gC)$$ which sends some element $c \in G_3$ to
the map $\sigma_c : \gC \to \gC$, which is defined via

	$$
\sigma_c( g	)=\left\{\begin{array}{ll}
c \cdot g &\mbox{if } g \in G_3 \\
\varphi_{2,3} (c)	&\mbox{if }  g \in G_2\\
\varphi_{1,3} (c)	&\mbox{if }  g \in G_1
\end{array}\right.
$$
We have to show that $\theta$ is well-defined. To see this, let $c \in G_3$. First, we show that $\theta(c)$ is a homomorphism of $\gC$. We consider to $a\in G_3$, and show $\sigma_c$ behaves well with respect to $F^{\gC}_{a,3}$, in the following sense:
$$
\sigma_c(F^{\gC}_{a,3}(b))=c \cdot b \cdot a= \sigma_c(b) \cdot a=F^{\gC}_{a,3}( \sigma_c(b)),
$$where, $b \in G_3$.
Furthermore, it behaves well with respect to $F_{2,3}^{\gC}$:

$$\begin{array}{ll}
\sigma_c(F_{2,3}^{\gC}(a))&=\sigma_c(\varphi_{2,3}(a))\\
&=\varphi_{2,3}(c) \cdot \varphi_{2,3}(a)\\
&=\varphi_{2,3}(c \cdot a)\\
&= \varphi_{2,3}( \sigma_c(a))\\
&=F_{2,3}^{\gC}(  \sigma_c(a)).
\end{array}$$

We apply similar arguments for $a, b \in G_i$ with $i = 1, 2$. Hence, $\sigma_c$ is indeed a homomorphism of $\gC$. It is straightforward to verify that each $\theta(c)$ is a bijection, as it corresponds to left translation in the groups. Thus, $\theta(c) \in \aut(\gC)$, and $\theta$ is well-defined.

Now, we show that $\theta$ is an isomorphism of groups. To verify that $\theta$ is a homomorphism, it must preserve the group operation. Let $c_1, c_2 \in G_3$ be arbitrary. Then, for any $a \in G_3$, we have
\[
\theta(c_1 \cdot c_2)(a) = c_1 \cdot c_2 \cdot a = \theta(c_1)(c_2 \cdot a) = (\theta(c_1) \theta(c_2))(a),
\]
and similarly for $a \in G_i$ with $i = 1, 2$.

Furthermore, $\theta$ respects inverses, as
\[
\theta(c^{-1}) = \theta(c)^{-1}.
\]

To show injectivity, suppose $\theta(c) = \theta(c')$ for some $c, c' \in G_3$. Then, for all $a \in G_3$,
\[
\sigma_c(a) = c \cdot a = c' \cdot a = \sigma_{c'}(a).
\]
Since left multiplication is injective, it follows that $c = c'$, proving that $\theta$ is injective.

For surjectivity, consider an arbitrary automorphism $\sigma \in \Aut(\gC)$. Setting $c = \sigma(e_{G_3})$, we claim that
\[
\sigma = \sigma_c.
\]
This follows from the identity
\[
\sigma(F^{\gC}_a(e_{G_3})) = F_a^{\gC}(\sigma(e_{G_3})).
\]
Since $\sigma$ is an arbitrary automorphism, this shows that every element of $\Aut(\gC)$ is of the form $\theta(c)$ for some $c \in G_3$. Thus, $\theta: G_3 \to \Aut(\gC)$ is an isomorphism of groups.

To establish clause (ii), it suffices to show that the maps $\varphi_{i,j}$ correspond to the natural restriction maps between the automorphism groups. This follows directly from our earlier observations: every automorphism of $\gC$ corresponds to right translation on each sort. Specifically, if the restriction of $\sigma$ to the $j^{th}$ sort is given by translation by $b \in G_j$, then on the $i^{th}$ sort, it translates by $\varphi_{i,j}(b)$.

 By Proposition \ref{a17} and the definition of the 3-sorted model $\gC$, we observe that $\varphi_{\bfc_{1,3}}$ has no weak lifting, while $\varphi_{\bfc_{2,3}}$ admits a lifting. Consequently, Corollary \ref{a14} applies, implying that $\bfc_{1,2}$ is not $\chi$-solvable in some forcing extension. Thus, by setting $\bfc := \bfc_{1,2}$, the corollary follows.
\end{proof}

Now, we are ready to prove the main result of this section.
\begin{theorem}
\label{a23}
Let $\chi$ be a cardinal and $\bfc$ be a uni-construction problem. If $\bfc$   has no lifting, then  in some forcing extension, $\bfc$ has no
$\chi$-solution.
\end{theorem}
\begin{proof}
Recall that $ G_{\bfc} := \aut(\gA_{\bfc})$ and $ H_{\bfc} := \aut(\gB_{\bfc})$. We set $G_1 := G_{\bfc}$, $G_2 := H_{\bfc}$, and $\varphi_{1,2} := \varphi_{\bfc}$. In light of Proposition \ref{a17}, we can find a group $G_3$ and a surjective homomorphism $\varphi_{2,3}: G_3 \twoheadrightarrow G_2$, which has a lifting such that the homomorphism $\varphi_{1,3} = \varphi_{1,2} \circ \varphi_{2,3} \in \Hom(G_3, G_1)$ has no weak lifting. Without loss of generality, we may assume that
$$a \in \gB_{\bfc} \wedge b \in G_3 \Rightarrow a \neq b.$$
Let $\langle a_\alpha : \alpha < \alpha_* \rangle$ list the elements of $\gB_{\bfc}$. We define a 3-sorted model $\gC$ as follows:

\begin{enumerate}
\item[$(\ast)$]
\begin{enumerate}
\item[(a)]  $\sort_{1,2}(\gC) = \gB_{\bfc}$ so $\sort_1(\gC) =
  \gA_{\bfc}$;

\item[(b)]  the set of elements of $\sort_3(\gC)$ is the set of
 elements of $G_3$;

\item[(c)] $F^{\gC}_{1,\alpha}:G_3 \rightarrow \sort_1(\gC)$ is defined
by the help of $ \varphi_{1,3}$. More precisely,
for any $b \in G_3$,   we set $ F^{\gC}_{1,\alpha}(b) =
(\varphi_{1,3}(b))(a_\alpha)$;

\item[(d)]

 $F^{\gC}_{2,\alpha}:G_3 \rightarrow \gB_{\bfc}$ is defined
  by:
\[
b \in G_3 \Rightarrow F^{\gC}_{2,\alpha}(b) =
(\varphi_{2,3}(b))(a_\alpha);
\]

\item[(e)]  $F^{\gC}_{3,c}:G_3 \rightarrow G_3$, where $c\in G_3$, is defined by:
\[
b \in G_3 \Rightarrow F^{\gC}_{3,c}(b) = cb.
\]
\end{enumerate}
\end{enumerate}
In the same vein as in the proof of Corollary \ref{a20}, we observe that
\begin{enumerate}
\item[$(\ast\ast)$]
\begin{enumerate}
	\item[(a)]  $\gA_{\bfc_{1,2}} = \sort_1(\gC),\gB_{\bfc_{1,2}} =
	\sort_{1,2}(\gC)$,
	
	\item[(b)]  $\gA_{\bfc_{2,3}} = \sort_{1,2}(\gC),\gB_{\bfc_{2,3}} = \gC$,
	
	\item[(c)]  $\gA_{\bfc_{1,3}} = \sort_1(\gC),\gB_{\bfc_{1,3}} = \gC$.
\end{enumerate}
\end{enumerate}
By our  construction, it is easily seen that $\varphi_{\bfc_{1,3}}$ has no weak lifting,
and
    $\varphi_{\bfc_{2,3}}$ has a lifting.

In view of Corollary \ref{a14}, we are able to find a forcing extension $V[G]$ of the universe in which
$\bfc_{1, 2}$ has no $\chi$-solution. But since $\bfc_{1, 2} = \bfc$, it follows that
$\bfc$ has no $\chi$-solution in $V[G]$. Thus, the theorem is proved.
\end{proof}

\section {A global consistency result} 
Let $\bfc$ be a definable uni-construction problem via a formula $\theta$. Recall that there exists a class function $F:\mathcal K^1_{\bfc} \to \mathcal K^2_{\bfc}$, such that for each $\gB \in \mathcal K^2_{\bfc}$, $F(\sort_1(\gB)) = \gB$, and
$$ F(\gA) = \gB \Leftrightarrow \theta(\gA, \gB), $$
for all two-sorted models $\gB$ with $\gA := \sort_1(\gB)$.
\begin{definition}
	\label{yu11}Let $c$ be as above. We say $\bfc$ is uniformly  definable
	if there is a formula $\Theta$ so that$$\exists y\theta(y,\gB)\rightarrow\exists y\Theta(y,\gB)\wedge\forall y(\Theta(y,\gB)\rightarrow\theta(y,\gB)).$$
\end{definition}

\begin{definition}\label{f}For an infinite cardinal $\lambda$, let $\mathbb{S}_\lambda$ be the  forcing notion
\[
\mathbb{S}_\lambda=\{ p: \lambda^{++} \times \lambda^{++} \times \lambda^{++} \rightarrow 2: |p| \leq \lambda         \},
\]
ordered by reverse inclusion. \end{definition}Thus  $\mathbb{S}_\lambda$ is forcing equivalent to
$\text{Add}(\lambda^+, \lambda^{++})$, the Cohen forcing for adding $\lambda^{++}$-many Cohen subsets of
$\lambda^+.$ Furthermore, in view of \cite[Lemma 5.2]{Sh:301}, we observe that it is $\lambda^{+}$-closed and satisfies the $\lambda^{++}$-c.c.

\begin{fact}
	\label{local-forcing} (Hodges-Shelah) Let $M$ be an inner model of $\text{ZFC}+\text{GCH}$ and let
	$\lambda$ be an infinite cardinal of $M$. Let $\mathbb{Q}$ be the forcing notion
	$\mathbb{S}_\lambda$ as computed in $M$ and let $\bold G$ be $\mathbb{Q}$-generic over $M$. Then
	the following holds in $M[\bold G]$:
	\begin{enumerate}
		\item[$(*)_\lambda:$] suppose $\bold{c}$ is a uniformisable uni-construction problem, such that $\bold c$ is defined using parameters from $V$,
		$\mathcal{B}_{\bold c} \in V$ and $\mathcal{B}_{\bold c}$ and $\aut(\mathcal{B}_{\bold c})$
		have size $\leq \lambda.$ Then  $\bold{c}$ is weakly natural.
	\end{enumerate}
\end{fact}

\begin{proof}
This is in \cite[Theorem 5.1]{Sh:301}.
\end{proof}
In this section, we are going to prove a global version of this theorem, which removes both the cardinality assumption and the parameter assumption from the above result. The proof uses the reverse Easton iteration of forcing notions, where we refer to \cite{baum} and \cite[Chapter 21]{jech} for more details on this subject.

\begin{theorem}
	\label{global-forcing} Let $\bold{c}$ be a uniformisable uni-construction problem. There exists a $\text{GCH}$ and  cofinality preserving
	class generic extension $V[\bold G]$ of the universe in which    $\bold{c}$ is weakly natural.
\end{theorem}
\begin{proof}
For a given an infinite cardinal $\lambda$, let $\mathbb{S}_\lambda$ be as Definition \ref{f}.
	Let
	\[
	\mathbb{P}=\big\langle \langle  \mathbb{P}_\lambda: \lambda \in \Ord          \rangle, \langle  \name{\mathbb{Q}}_\lambda: \lambda \in \Ord    \rangle\big\rangle
	\]
	be the reverse Easton iteration of forcing notions, such that
	for each ordinal $\lambda, \name{\mathbb{Q}}_\lambda$ is forced to be the trivial forcing notion except
	$\lambda$ is an infinite cardinal, in which case we let
	\begin{center}
		$\Vdash_{\mathbb{P}_\lambda}$``$\name{\mathbb{Q}}_\lambda= \name{\mathbb{S}}_\lambda$''.
	\end{center}
	Thus, a condition in $\mathbb{P}$ is a partial function $p$ such that:
	\begin{enumerate}
		\item $\dom(p)$ is a set of ordinals,
		
		\item if $\lambda \in \dom(p),$ then $p \upharpoonright \lambda \Vdash_{\mathbb{P}_\lambda}$``$p(\lambda) \in \name{\mathbb{Q}}_\lambda$'',
		
		\item for any regular cardinal $\kappa$,
		$|\supp(p) \cap \kappa| < \kappa$, where
		\[
		\supp(p):=\big\{\lambda \in \dom(p): p \upharpoonright \lambda \Vdash_{\mathbb{P}_\lambda} \text{``}p(\lambda) \neq 1_{\name{\mathbb{Q}}_\lambda} \text{''}    \big\}.
		\]
		
	\end{enumerate}
	Let also
	\begin{center}
		$\bold G = \big\langle \langle   \bold G_\lambda:  \lambda \in \Ord  \rangle, \langle  \bold H_\lambda:  \lambda \in \Ord  \rangle\big\rangle$
	\end{center}
	be $\mathbb{P}$-generic over $V$. Thus for each infinite cardinal $\lambda$,
	$\bold G_\lambda = \mathbb{P}_\lambda \cap \bold G$ is $\mathbb{P}_\lambda$-generic over $V$ and $H_\lambda$ is
	$ \name{\mathbb{S}}_\lambda[\bold G_\lambda]$-generic over $V[\bold G_\lambda]$.
	We are going to show that $V[\bold G]$ is as required. To this end, let us recall the following well-known result from \cite{baum} (also see \cite[Chapter 21]{jech}).
	\begin{fact}
		\label{lem1}Adopt the above notation. Then the following assertions are valid:
		\begin{enumerate}
			%\item The forcing notion $\mathbb{P}$ is tame, in particular $V[\bold G]$ is a model of $\text{ZFC}$,
			
			\item $V[\bold G]$ is a $\text{GCH}$ and  cofinality preserving
			class generic extension of $V$,
			
			\item  $V[\bold G]$ and $ V[\bold G_\lambda]$ contain the same $\lambda$-sequences of ordinals.
		\end{enumerate}
	\end{fact}
	%In particular, it follows that $V[\bold G]$ is a cardinal and cofinality preserving generic extension of $V$.
	
	Recall that $\bold c$ is a uniformisable uni-construction problem. Let $\lambda$ be a large enough  cardinal such that
	$\mathcal{B}_{\bold c}$, $\aut(\mathcal{B}_{\bold c})$ and the parameters occurring in the definition of
	$\bold c$ and the formula uniformizing it are all in $ V[\bold G_\lambda]$, and $|\mathcal{B}_{\bold c}|, |\aut(\mathcal{B}_{\bold c})| \leq \lambda$.
	By Fact $\ref{local-forcing},$ applied to the model  $V[\bold G_\lambda]$ and the
	uni-construction problem $\bold c,$ we   conclude that
	\begin{center}
		$V[\bold G_\lambda][H_\lambda] \models$``$\bold c$ is weakly natural''.
	\end{center}
	On the other hand, $V[\bold G]$ is a  generic extension of $V[\bold G_\lambda][H_\lambda]$ by a class forcing notion which adds no new subsets to
	$\lambda^+.$  Thus
	\begin{center}
		$V[\bold G] \models$``$\bold c$ is weakly natural''.
	\end{center}
Thus, the theorem is proved.
\end{proof}

\section{Uniformity} 

In this section, we address Problem \ref{p4}. Our main result is Theorem \ref{d2}, which can be viewed as a generalization of \cite[Theorem 3]{Sh:160}. We present our result in the context of two-sorted models.

\begin{theorem}
	\label{d2}
Let $\tau$ be a
vocabulary, and let
$\mathcal K$ be a class of $\tau$-models such that

	\begin{enumerate}
	\item[(i)]  $\mathcal K$ is first order definable
	from a parameter $\bfp$,
	
	\item[(ii)]  every $\gB \in \mathcal K$ is two sorted,	\item[(iii)]  the natural homomorphism $\varphi_{\gB}:\aut(\gB)\twoheadrightarrow\aut(\sort_1(\gB))$ splits.
\end{enumerate}
 Then
	there exists a  class
	function $F$, which is uniformly definable from
	the parameter $\bfp$ and furnished the following assertions:
	\begin{enumerate}
		\item[(a)]  the domain of $F$ is $   \mathcal K_1 = \{\sort_1(\gB):\gB \in \mathcal K\}$,
		
		\item[(b)]  if $\gA \in \mathcal K_1$, then $F(\gA) \in \mathcal K$ and $\sort_1(F(\gA)) = \gA$.
	\end{enumerate}
	\iffalse
	If (A) then (B) where:
	\mn
	\begin{enumerate}
		\item[(A)]
		\begin{enumerate}
			\item[(a)]  $\tau$ is a vocabulary with 2-sorted so for a $\tau$-model
			$\gB,\gA = \sort_1(\gB)$ is a $\tau_1$-model, for suitable $\tau_1$
			\sn
			\item[(b)]  $\mathcal K$ is a class of $\tau$-models so first order definable
			from the parameter $\bfp$
			\sn
			\item[(c)]  if $\gB \in \mathcal K$ then the natural homomorphism $\pi^*_{\gB}$
			from $\aut(\gB)$ into $\aut(\sort_1(\gB))$ onto $\aut(\sort_1(\gB))$ does splits
		\end{enumerate}
		\sn
		\item[(B)]  using the parameter $\bfp$ uniformly we can define a class
		function $F$ such that:
		\sn
		\begin{enumerate}
			\item[(a)]  the domain is $\mathcal K_1 = \{\sort_1(\gB):\gB \in\mathcal K\}$
			\sn
			\item[(b)]  if $\gA \in \mathcal K_1$ then $F(\gA) \in \mathcal K$ and $\sort_1(F(\gA)) = \gA$.
		\end{enumerate}
	\end{enumerate}
	\fi
\end{theorem}

\begin{proof}
	We prove the theorem through a sequence of claims.
	First, we show that the problem can be reduced to the case where $\mathcal{K}$ has only one equivalence class.
	To this end, we define the class function $\bfH$ with $\dom(\bfH) = \mathcal{K}$, as follows:
	\begin{center}
		$\bfH(\gB) := \bigg\{\gB':\gB'$ is a $\tau$-model isomorphic to
		$\gB$  with universe equal to the cardinal $\|\gB\|\bigg\}$.
	\end{center}
Clearly, $\bfH$ is definable from the parameter $\bfp$. For every $\bfx \in \Range(\bfH)$, we set
	
	\begin{enumerate}
		\item   $\mathcal K_{\bfx} := \{\gB \in \mathcal K:\bfH(\gB) = \bfx\}$,
		
		\item   $\mathcal K_{\bfx,1}: = \{\sort_1(\gB):\gB \in \mathcal K_{\bfx}\}$.
	\end{enumerate}
	
We remark that $\langle K_\bfx: \bfx \in \Rang(\bfH)\rangle$ is a partition of $\mathcal K$, which is uniformly definable using the parameter $\bfp$, so it suffices to uniformly deal with $\mathcal K_{\bfx}$ for each ${\bfx} \in \Range(\bfH)$. Thus, let us fix some ${\bfx} \in \Range(\bfH)$.
Now, we turn our attention to $\Upsilon$. By definition, it is the family of all pairs $(\gB, \psi)$ equipped with the following properties:
\begin{enumerate}
	\item[$\Upsilon_1$)] $\gB \in \mathcal K$ is such that for some $\gB' \in \bfx$, the function $b \mapsto (\gB', b)$, for $b \in \gB'$, is an isomorphism from $\gB'$ onto $\gB$.
	\item[$\Upsilon_2$)] $\psi$ is a weak lifting corresponding to $\gB$, more precisely, a lifting for $\varphi_\gB$.
\end{enumerate}

Let us collect all of them under a new name:
\begin{enumerate}
	\item[(3)] Let $\bfy_{\bfx} := \big\{(\gB, \psi): (\gB, \psi) \in \Upsilon \big\}$.
\end{enumerate}

To simplify things, we present a series of claims (see Claims \ref{cla1}-\ref{cla6}):
	\begin{claim}
		\label{cla1} Let $\bfy_{\bfx}$ be as $(3)$.
	Then	$\bfy_{\bfx}$ is non-empty.
	\end{claim}
	\begin{proof}
		By our hypothesis, for each $\gB \in \mathcal K$,  there is a weak lifting
		for $\varphi_\gB$, and hence $\bfy_{\bfx}$ is non-empty, as requested.
	\end{proof}
	Let $S = \bfy_{\bfx},$ and for each $s \in S$ set $(\gB_s,\psi_s)=s$. It then follows that
	 $\bfy_{\bfx} = \{(\gB_s,\psi_s):s \in S\}.$
	For any $s \in S$, we set
	\begin{enumerate}
		\item[(4)]
		$\gA_s :=  \sort_1(\gB_s).$
	\end{enumerate}
	By replacing each $\gB_s$ with $\gB_s \times \{s\}$ if necessary, we may assume that the $\gB_s$'s are pairwise disjoint.
	For each $b \in \bigcup\limits_{s \in S} \gB_s$, we define
	\[
	r(b) := \text{the unique } s \in S \text{ such that } b \in \gB_s.
	\]
	\begin{claim}
		\label{cla2}
		$\bfx$ is definable from $\bfy_{\bfx}$.
	\end{claim}
	\begin{proof}
		Recall that $\bfx$ is equal to $$\bigg\{\gB': \exists (\gB,\psi) \in \bfy_{\bfx}  \big(\emph{the function } b \mapsto (\gB',b)
		\emph{ induces }   \gB'\stackrel{\cong}\longrightarrow\gB \big)  \bigg\},$$ which establishes the result.
	\end{proof}

 Let $\big\langle h_{s, t}:\gA_s\stackrel{\cong}\longrightarrow \gA_t\big\rangle_{s, t
	\in S}$ be a family of isomorphisms. Let us introduce the following terminology .	We say the  $\big\langle h_{s, t}:\gA_s\stackrel{\cong}\longrightarrow \gA_t\big\rangle_{s, t
		\in S}$  is a commutative $\gA_\bullet$-family (or a commutative family over $\gA_\bullet$), provided:
	
		\begin{enumerate}
		\item[$\bullet$] $h_{r,s} \circ h_{s,t} =
		h_{r,t}$ and
	\item[$\bullet$]  $h_{s,s} = \id_{\gA_s}$,
	\end{enumerate}
	where $r,s,t \in S  $. In other terms, we have  the following commutative diagram: $$\xymatrix{
		&  \gA_s\ar[rr]^{h_{s,t}}&&\gA_t\\
		&& \gA_r\ar[ul]^{h_{r,s}}\ar[ur]_{h_{r,t}}
		&&&}$$
	\begin{enumerate}
		\item[(5)]By the first automorphism class of $\bfy_{\bfx}$
		 we mean  $$\aut_1(\bfy_{\bfx}) :=\big \{\bar h = \langle h_{s, t}:s, t
		\in S\rangle:\bar h   \emph{  is a commutative family over } \gB_\bullet\big\}  .$$		
In the same vein, we say the family $\langle f_{s, t}:\gB_s\stackrel{\cong}\longrightarrow \gB_t\rangle_{s, t
			\in S}$ of   isomorphisms is a commutative $\gB_\bullet$-family provided
	$f_{r,s} \circ f_{s,t} =
		f_{r,t}$ and $f_{s,s} = \id_{\gB_s}$ for all $r,s,t \in S  $. So, the diagram $$\xymatrix{
			&  \gB_s\ar[rr]^{f_{s,t}}&&\gB_t\\
			&& \gB_r\ar[ul]^{f_{r,s}}\ar[ur]_{f_{r,t}}
			&&&}$$is commutative.

		\item[(6)] 	By the second automorphism class of $\bfy_{\bfx}$, we mean as follows:$$\aut_2(\bfy_{\bfx}):=\big \{\bar f = \langle f_{s, t}:s, t
		\in S\rangle: \bar f  \emph{  is a commutative family over } \gB_\bullet\big\}.$$
	\end{enumerate}

	\begin{enumerate}
		\item[(7)]	For $\gA \in\mathcal K_{\bfx,1}$, the isomorphism class of $\gA$ with respect to $\bfy_{\bfx}$, is defined by:\begin{center}
		$\iso_{\bfy_{\bfx}}(\gA) :=\{\bar\pi:\bar\pi =
		\langle \pi_s:s \in S\rangle \emph{ such that } \pi_s: \gA_s\stackrel{\cong}\longrightarrow\gA\}. $\end{center}
	\end{enumerate}
	Suppose $\bar\pi \in \iso_{\bfy_{\bfx}}(\gA)$ and $s,t \in S$. We define 	$h_{\bar\pi,s,t}$ and $\bar h_{\bar\pi}$ as follows:
	\begin{enumerate}
		\item[(8)]
		$h_{\bar\pi,s,t} = \pi^{-1}_t \circ \pi_s
		\in \iso(\gA_s,\gA_t),$
		
		\item[(9)] $\bar h_{\bar\pi}=  \langle h_{\bar\pi,s,t}: s, t \in S \rangle.$
	\end{enumerate}
	This property can be summarized by the following  diagram:
	$$\xymatrix{
		&  \gA_s\ar[dr]_{\pi_s}\ar[rr]^{ h_{\bar\pi,s,t}}&&\gA_t\\
		&& \gA\ar[ur]_{\pi^{-1}_t}
		&&&}$$
	The proof of next claim is evident.
	\begin{claim}
		Suppose $\gA \in \mathcal K_{\bfx,1}$ and $\bar\pi \in \iso_{\bfy_{\bfx}}(\gA)$. Then $\bar h_{\bar\pi} \in \aut_1(\bfy_{\bfx}).$
	\end{claim}
	
Recall that $\psi_s$ is a weak lifting corresponding to $\gB_s$.
For $\gA \in \mathcal{K}_{\bfx,1}$, suppose that $\bar{\pi} \in \iso_{\bf y_{\bfx}}(\gA)$ and
$\bar{g} \in \aut_2(\bfy_{\bfx})$ satisfy the conditions
$g_{s,s} = \psi_s(h_{\bar{\pi},s,s})$ and $h_{\bar{\pi},s,t} \subseteq g_{s,t}$ for all $s,t \in S$.
Adopt the previous notation, and let  $\bar{b} = \langle b_s : s \in S \rangle$.
	We say the triple $(\bar\pi,\bar g,\bar b)$ is \emph{matched},   if $\bar\pi,$
	$\bar g$ and $\bar b$ are as above,
and 	$b_s \in \gB_s$ we have
	$g_{s,t}(b_t) = b_s$ for $s,t \in S$.

Let $\Sigma(\bfx, \gA)$ denote the family of all matched triples $(\bar\pi,\bar g,\bar b)$.
 The following diagram summarizes the above situation:
	$$\xymatrix{
		&b_s\in\gB_s\ar[rr]^{g_{s,t}}&&\gB_t\ni b_t	\\&  \gA_s\ar[dr]_{\pi_s}\ar[rr]^{ h_{\bar\pi,s,t}}\ar[u]^{\subseteq}&&\gA_t\ar[u]_{\subseteq}\\
		&& \gA\ar[ur]_{\pi^{-1}_t}
		&&&}$$In order to define the ``universal model'', we first consider to:
	\begin{enumerate}
		\item[(10)]
		$X_{\bfy_{\bfx}}(\gA) := \big\{(\bar\pi,\bar g,\bar b):(\bar\pi,\bar g,\bar b) \in
		\Sigma(\bfx, \gA)\big\}$.
	\end{enumerate}
For $\gA \in\mathcal K_{\bfx, 1}$, let $E_{\gA} $ be
	the following two place relation on $X_{\bfy_{\bfx}}(\gA)$:
	$$(\bar\pi,\bar g_1,\bar b_1) E_{\gA} (\pi_2,\bar
	g_2,\bar b_2)$$
	if and only if the following two conditions are valid:
	\begin{itemize}
		\item  $(\bar\pi_1,\bar g_1,\bar b_1)$ and $(\bar\pi_2,\bar g_2,\bar
		b_2)$  belong to $X_{\bfy_{\bfx}}(\gA)$,
		\sn
		\item  letting $\widetilde{\pi}_s = \pi^{-1}_{1,s} \pi_{2,s}$ and $h_s =
		\psi_s(\widetilde{\pi}_s^{-1})$, we have $h_s(b_{1,s}) = b_{2,s}$.
	\end{itemize}
	Note that  $\widetilde{\pi}_s$ is an
	automorphism of $\gA_s$, hence so is $\widetilde{\pi}_s^{-1}$. Consequently,  $h_s =
	\psi_s(\widetilde{\pi}_s^{-1})$ is a member of $H_s = \aut(\gB_s)$ which extends $\widetilde{\pi}_s$. Since $b_{1,2},b_{2,s} \in \gB_s$ so $h_s(b_{1,s})=b_{2,s} \in
	\gB_s$ is meaningful.
	
	\begin{claim}Let $E_{\gA} $ be
		the above relation on $X_{\bfy_{\bfx}}(\gA)$.
	Then	$E_{\gA}$ is an equivalence relation.
	\end{claim}
	\begin{proof}
		It is clearly reflexive and symmetric. To show that  $E_{\gA}$ is transitive, assume that
		\begin{enumerate}
			\item[(a)]  $(\bar\pi_1,\bar g_1,\bar b_1) E_{\gA} (\bar\pi_2,\bar
			g_2,\bar b_2)$,
			
			\item[(b)]  $(\bar\pi_2,\bar g_2,\bar b_2) E_{\gA} (\bar\pi_3,\bar g_3,
			\bar b_3)$.
		\end{enumerate}
	According to the definition of $E_{\gA}$, we have:

		\begin{enumerate}
			\item[(c)]  $(\psi_s(\pi^{-1}_{1,s} \pi_{2,s})^{-1})(b_{1,s}) = b_{2,s}$ for $s \in S$
			
			\item[(d)]  $(\psi_s(\pi^{-1}_{2,s} \pi_{3,s})^{-1})(b_{2,s}) =
			b_{3,s}$ for $s \in S$.
		\end{enumerate}
		Hence
		$$\begin{array}{ll}
		b_{3,s}&=(\psi_s(\pi^{-1}_{2,s} \pi_{3,s})^{-1})(b_{2,s})\\
		&= (\psi_s(\pi^{-1}_{2,s} \pi_{3,s})^{-1})((\psi_s(\pi^{-1}_{1,s} \pi_{2,s})^{-1})(b_{1,s}))\\
		&=(\psi_s \big( (\pi^{-1}_{2,s} \pi_{3,s})^{-1} \circ  (\pi^{-1}_{1,s} \pi_{2,s})^{-1}  \big)(b_{1,s}))\\
		&=\psi_s((\pi^{-1}_{1,s} \pi_{3,s})^{-1})(b_{1,s}),
		\end{array}$$
		as requested.
	\end{proof}
For simplicity, we assume that the structures $\gB$ are relational, so they only contain relations. This is simply because every function can be identified by a relation, satisfying some extra conditions.

	We define a model $\gB'_{\gA}$ as follows:
	\begin{enumerate}
		\item[(11.1)] the universe of $\gB'_{\gA}$ is $X_{\bfy_{\bfx}}(\gA)$,
		
	\item[(11.2)] For any $R \in \tau$ which is an $n$-place relation, we set $$R^{\gB'_{\gA}}:=\bigg\{ \big\langle  (\bar\pi_\ell,\bar g_\ell,\bar b_\ell): \ell < n   \big\rangle \in X_{\bfy_{\bfx}}(\gA)^n:
		\exists s \in S  \emph{ such that  } \langle b_{\ell, s}: \ell < n    \rangle \in R^{\gB_s}         \bigg \}.$$
	\end{enumerate}
	\begin{claim}
		\label{cla3}
		In the definition of $R^{\gB'_{\gA}}$, we can replace ``$\exists s \in S$'' by ``$\forall s \in S$''.
	\end{claim}
	\begin{proof}
		Suppose $s, t \in S$ and $\langle (\bar\pi_\ell,\bar g_\ell,\bar b_\ell): \ell < n   \rangle \in X_{\bfy_{\bfx}}(\gA)^n$. We have to show that
		\[
		\langle b_{\ell, s}: \ell < n    \rangle \in R^{\gB_s} \Leftrightarrow \langle b_{\ell, t}: \ell < n    \rangle \in R^{\gB_t}.
		\]
		This follows from the fact that $g_{s, t}: \gB_s \rightarrow \gB_t$ is an isomorphism and $g_{s,t}(b_{\ell, s})=b_{\ell, t}$.
	\end{proof}
	\begin{claim}
		\label{cla4} Let $\iota=1, 2$ and $n$ be any integer.
		Assume $(\bar\pi^\iota_\ell,\bar g^\iota_\ell,\bar b^\iota_\ell)\in X_{\bfy_{\bfx}}(\gA)$
		and  $(\bar\pi^1_\ell,\bar g^1_\ell,\bar b^1_\ell) E_{\gA} (\bar\pi^2_\ell,\bar g^2_\ell,\bar b^2_\ell) $  for all $ \ell < n$.
		Then
		\[
		\langle (\bar\pi^1_\ell,\bar g^1_\ell,\bar b^1_\ell): \ell < n      \rangle \in R^{\gB'_{\gA}} \Leftrightarrow \langle (\bar\pi^2_\ell,\bar g^2_\ell,\bar b^2_\ell): \ell < n      \rangle \in R^{\gB'_{\gA}}.
		\]
	\end{claim}
	\begin{proof}
	Assume that $\langle (\bar\pi^1_\ell, \bar g^1_\ell, \bar b^1_\ell) : \ell < n \rangle \in R^{\gB'_{\gA}}$.
	For each $\ell < n$, we have
	\[
	(\bar\pi^1_\ell, \bar g^1_\ell, \bar b^1_\ell) E_{\gA} (\bar\pi^2_\ell, \bar g^2_\ell, \bar b^2_\ell).
	\]
	This, in turn, implies that
	\[
	\psi_s((\pi^{2}_{\ell, s})^{-1} \pi^1_{\ell, s})(b^1_{\ell, s}) = b^2_{\ell, s},
	\]
	where $s \in S$.
	
	Since $\psi_s((\pi^{2}_{\ell, s})^{-1} \pi^1_{\ell, s})$ is an automorphism of $\gB_s$, we obtain
	\[
	\langle b^1_{\ell, s} : \ell < n \rangle \in R^{\gB_s} \Leftrightarrow \langle b^2_{\ell, s} : \ell < n \rangle \in R^{\gB_s}.
	\]
	Thus, we conclude that
	\[
	\begin{array}{ll}
	\big\langle (\bar\pi^1_\ell, \bar g^1_\ell, \bar b^1_\ell) : \ell < n \big\rangle \in R^{\gB'_{\gA}}
	&\Rightarrow \exists s \in S \text{ such that } \langle b^1_{\ell, s} : \ell < n \rangle \in R^{\gB_s} \\
	&\Rightarrow \langle b^2_{\ell, s} : \ell < n \rangle \in R^{\gB_s} \\
	&\Rightarrow \big\langle (\bar\pi^2_\ell, \bar g^2_\ell, \bar b^2_\ell) : \ell < n \big\rangle \in R^{\gB'_{\gA}}.
	\end{array}
	\]
	By symmetry, we also obtain
	\[
	\big\langle (\bar\pi^2_\ell, \bar g^2_\ell, \bar b^2_\ell) : \ell < n \big\rangle \in R^{\gB'_{\gA}}
	\Rightarrow \big\langle (\bar\pi^1_\ell, \bar g^1_\ell, \bar b^1_\ell) : \ell < n \big\rangle \in R^{\gB'_{\gA}}.
	\]
	Thus, the claim follows.
	\end{proof}
	In   light of Claim \ref{cla4} we observe that $E_{\gA}$ is a congruence relation on $\gB'_{\gA}$. Here, we define a function $\bold k$ with domain $\gA$
	as follows:
	\[
	\bold k(a) = \big\{ (\bar\pi,\bar g,\bar b) \in X_{\bfy_{\bfx}}(\gA): b_s=\pi_s^{-1}(a) \text{~for all~} s \in S  \big \}.
	\]
	\begin{claim}
		\label{cla5} If $a \in \gA$, then $\bold k(a)$ is an $E_{\gA}$-equivalence class.
	\end{claim}
	\begin{proof}
	Let us first show that $\bold{k}$ is closed under the relation $E_{\gA}$, in the sense that
	$(\bar\pi_2, \bar g_2, \bar b_2) \in \bold{k}(a)$
	provided that $(\bar\pi_1, \bar g_1, \bar b_1) \in \bold{k}(a)$
	and $(\bar\pi_1, \bar g_1, \bar b_1) E_{\gA} (\bar\pi_2, \bar g_2, \bar b_2).$
	
	To this end, let $s \in S$ and recall that
	\begin{itemize}
		\item $b_{1, s} = \pi_{1, s}^{-1}(a),$
		\item $\psi_s(\pi^{-1}_{2,s} \pi_{1,s})(b_{1,s}) = b_{2,s}.$
	\end{itemize}
	
	It then follows that
	\[
	\begin{array}{ll}
	b_{2, s} &= \psi_s(\pi^{-1}_{2,s} \pi_{1,s})(b_{1,s})\\
	&= \psi_s(\pi^{-1}_{2,s} \pi_{1,s})(\pi_{1, s}^{-1}(a))\\
	&= \pi^{-1}_{2,s} \pi_{1,s} \pi_{1, s}^{-1}(a)\\
	&= \pi^{-1}_{2,s}(a).
	\end{array}
	\]
	In sum, $(\bar\pi_2, \bar g_2, \bar b_2) \in \bold{k}(a)$, as required.
	
	Next, we show that
	\[
	(\bar\pi_1, \bar g_1, \bar b_1) E_{\gA} (\bar\pi_2, \bar g_2, \bar b_2)
	\]
	provided that $(\bar\pi_1, \bar g_1, \bar b_1), (\bar\pi_2, \bar g_2, \bar b_2) \in \bold{k}(a)$.
	
	Let $s \in S$ and set $h_s = \psi_s(\pi^{-1}_{2,s} \pi_{1,s})$.
	We have to show that $h_s(b_{1, s}) = b_{2, s}$:
	\[
	\begin{array}{ll}
	h_s(b_{1, s}) &= \psi_s(\pi^{-1}_{2,s} \pi_{1,s})(b_{1, s})\\
	&= \psi_s(\pi^{-1}_{2,s} \pi_{1,s})(\pi_{1,s}^{-1}(a))\\
	&= \pi^{-1}_{2,s} \pi_{1,s} \pi_{1,s}^{-1}(a)\\
	&= \pi^{-1}_{2,s}(a)\\
	&= b_{2,s}.
	\end{array}
	\]
	Thus, the proof is complete.
	\end{proof}
Now, we are ready to introduce the universal model:
	\begin{claim}
		\label{cla6}
		$\gB'_{\gA}/ E_{\gA}$ is isomorphic to $\gB_s,$ for $s \in S.$
	\end{claim}
	\begin{proof}
Fix $s \in S$, and let $\phi: \gA_s \simeq \gA$ be an isomorphism. Set
		$$Y_{s,\phi}:=\{(\bar\pi,\bar g,\bar b) \in X_{\bfy_{\bfx}}(\gA): \pi_s=\phi      \},$$ and define a function
		\[
		\rho_{s, \phi}: Y_{s,\phi} \rightarrow \gB_s
		\]
		as $\rho_{s, \phi}((\bar\pi,\bar g,\bar b))=b_s$. Clearly,
		$\rho_{s, \phi}$ is  well-defined. Next, we bring the following two auxiliary observations:
		\begin{itemize} {\label{a9.1}}
			\item[$(*)_{5.11.1}$] Suppose $x_1=(\bar\pi_1,\bar g_1,\bar b_1)$ and $x_2=(\bar\pi_2,\bar g_2,\bar b_2)$ are in $Y_{s,\phi}$. Then
			\[
			x_1 E_{\gA} x_2 \Rightarrow \rho_{s, \phi}(x_1) = \rho_{s, \phi}(x_2).
			\]
		\end{itemize}
Indeed, suppose that $x_1 E_{\gA} x_2$. It follows that
\[
\psi_s(\pi^{-1}_{2,s} \pi_{1,s})(b_{1, s}) = b_{2, s}.
\]
But as $\pi_{1,s} = \phi = \pi_{2, s}$, we have
\[
\psi_s(\pi^{-1}_{2,s} \pi_{1,s})(b_{1, s}) = \psi_s(\id_{\gA_s})(b_{1, s}) = \id_{\gB_s}(b_{1, s}) = b_{1, s},
\]
and thus $\rho_{s, \phi}(x_1) = b_{1, s} = b_{2, s} = \rho_{s, \phi}(x_2)$.
This completes the proof of $(*)_{5.11.1}$.

Suppose $\bar\pi_2 = \langle \pi^2_s : s \in S \rangle \in \iso_{\bfy_{\bfx}}(\gA)$ and let
$x_1 = (\bar\pi_1, \bar g_1, \bar b_1) \in X_{\bfy_{\bfx}}(\gA)$.
Recall that $\widetilde{\pi}_s = \pi^{-1}_{1, s} \pi_{2, s}$.
For simplicity, we set:
\begin{itemize}
	\item $\widetilde{\Pi}_s := \psi_s(\widetilde{\pi}_s) \in \aut(\gB_s)$.
\end{itemize}
	
	The second auxiliary observation is:

		\begin{itemize}{\label{a9.2}}
			\item[$(*)_{5.11.2}$] Suppose $\bar\pi_*=\langle \pi^*_s: s \in S  \rangle \in \iso_{\bfy_{\bfx}}(\gA)$ and $s(*) \in S$. Let
			$x_1=(\bar\pi_1,\bar g_1,\bar b_1) \in X_{\bfy_{\bfx}}(\gA)$.  Then there is $x_2=(\bar\pi_2,\bar g_2,\bar b_2) \in [x_1]_{E_{\gA}}$
			such that
			\begin{itemize}
				\item[$\bullet$]  $\bar\pi_2=\bar\pi_*$ and
					\item[$\bullet$]  $b_{2, s(*)}=\widetilde{\Pi}_{s(*)}^{-1}b_{1, s(*)}$.
			\end{itemize}

		\end{itemize}
	[Why? In order to define $x_2$, we set
			\begin{itemize}
				\item[(i)] $\bar\pi_2=\bar\pi_*$,
				\item[(ii)] for $s, t \in S$, $g_{2, t, s}= \psi_t ((\pi^{-1}_{1, t} \pi_{2, t})^{-1})\circ g_{1, t, s} \circ \psi_s(\pi^{-1}_{1, s}\pi_{2, s})$.
			\iffalse	
				For
				simplicity, we set: \begin{itemize}
					\item $\Pi_s:=\pi^{-1}_{1, s}\pi_{2, s}$
					(resp.  $\Pi_t:=\pi^{-1}_{1, t}\pi_{2, t}$)
					\item  $\widetilde{\Pi}_s:=\psi_s(\pi^{-1}_{1, s}\pi_{2, s})$ (resp. $\widetilde{\Pi}_t:=\psi_t(\pi^{-1}_{1, t}\pi_{2, t})$).
				\end{itemize}
			\fi
			
The following   commutative diagram summarizes the above situation:
				
$$\quad\quad\quad\quad\quad\quad\quad\quad\xymatrix{
& 	 \gB_t\ar[rrr]^{g_{2, t, s}} &&&\gB_s	\\
					&& \gB_t\ar[ul]_{\widetilde{\Pi}_t}\ar[r]^{g_{1, t, s}} &\gB_s \ar[ur]^{\widetilde{\Pi}_s}\\
					&\gA_t\ar[uu]^{\subseteq}&&&\gA_s\ar[uu]_{\subseteq}\\
					&\gA \ar[u]^{\pi_{2, t}^{-1}}   &\ar[r]_{h_{1, t, s}}
					\gA_t \ar[l]^{\pi_{1, t}}\ar[ul]_{\widetilde{\pi}_t}\ar[uu]_{\subseteq}&
					\gA_s\ar[uu]^{\subseteq}\ar[r]_{\pi_{1, s}}\ar[ur]^{\widetilde{\pi}_s}  &\gA \ar[u]_{\pi_{2, s}^{-1}}
					&&&
&&}$$				
Let us depict the resulting   commutative diagram:
				$$\xymatrix{
					&&\gB_s	\\	&\gB_t\ar[dr]_{\widetilde{\Pi}_t^{-1}}\ar[ru]^{g_{2, t, s}}& &\gB_s \ar[ul]_{\widetilde{\Pi}_s}	\\& 	 &\gB_t\ar[ur]_{g_{1, t, s}}
					&&&}$$
			\iffalse	\item[(iii)] for $s \in S, b_{2, s}=g_{2, s, s(*)}(\widetilde{\Pi}_{s(*)}b_{1, s(*)})$.\fi
			\end{itemize}
			It is easily seen that $x_2$ as defined above is as required. Indeed, by (i), $\bar\pi_2=\bar\pi_*$.
			Also, it is clear that for each $s \in S$, $g_{2, s, s}=\id_{\gB_s}$. \iffalse
			and in particular, $$b_{2, s(*)}=g_{2, s(*), s(*)}(b_{2, s(*)})=g_{2, s(*), s(*)}(\widetilde{\Pi}_{s(*)}^{-1}b_{1, s(*)})=\widetilde{\Pi}_{s(*)}^{-1}g_{1, s(*), s(*)}(b_{1, s(*)})=\widetilde{\Pi}_{s(*)}^{-1}b_{1, s(*)}.$$
			\fi
	Next, we show   the following diagram $$\xymatrix{
				&  \gB_s\ar[rr]^{ g_{2,s, t}}&&\gB_t\\
				&& \gB_r\ar[ul]^{ g_{2,r, s}}\ar[ur]_{ g_{2,r, t}}
				&&&}$$is commutative where
			$r, s, t \in S$.
	 To check this, we first recall that
			$\widetilde{\pi}_s = \pi^{-1}_{1,s} \pi_{2,s}$, then
			we have
			$$\begin{array}{ll}
			g_{2,r, s} \circ g_{2, s, t}&= \big(\psi_r (\widetilde{\pi}_r)^{-1}\circ g_{1, r, s} \circ \psi_s(\widetilde{\pi}_s)\big) \circ
			\big(\psi_s (\widetilde{\pi}_s)^{-1}\circ g_{1, s, t} \circ \psi_t(\widetilde{\pi}_t)\big)\\
			&= \psi_r (\pi^{-1}_{1, r} \pi_{2, r})^{-1}\circ g_{1, r, s} \circ g_{1, s, t} \circ \psi_t(\pi^{-1}_{1, t}\pi_{2, t})\\
			&= \psi_r (\pi^{-1}_{1, r} \pi_{2, r})^{-1}\circ g_{1, r, t} \circ \psi_t(\pi^{-1}_{1, t}\pi_{2, t}) \\
			&= g_{2, r, t}.
			\end{array}$$Also,
				$$\begin{array}{ll}
			g_{2,s(\ast), s} (b_{2,s(\ast)})&= g_{2,s(\ast), s} (\widetilde{\Pi}_{s(\ast)}^{-1}b_{1,s(\ast)})\\
			&=\widetilde{\Pi}_{s}^{-1}g_{1,s(\ast), s}(b_{1,s(\ast)})\\
			&= \widetilde{\Pi}_{s}^{-1}(b_{1,s}) \\
			&=b_{2,s}.
			\end{array}$$
			In summary, we have proved that $$x_2=(\bar\pi_2,\bar g_2,\bar b_2) \in X_{\bfy_{\bfx}}(\gA).$$ Next, we are going to show that
			$x_2 E_{\gA} x_1.$ To this end, let $s \in S$, and set	\begin{itemize}
				\item  $\pi_s=\pi_{1, s}^{-1}\pi_{2, s}$,
				\item $h_s =\psi_s(\pi_s^{-1})$.
			\end{itemize}   Then:
			$$\begin{array}{ll}
			h_s(b_{1,s})&= \psi_s(\pi_{2, s}^{-1}\pi_{1, s})(b_{1,s})\\&=\widetilde{\Pi}_s(g_{1, s(*), s}(b_{1, s(*)}))\\&=g_{2, s, s(*)}(\widetilde{\Pi}_{s(*)}b_{1, s(*)})
\\&=g_{2, s, s(*)}(b_{2, s(*)})
\\
&=
b_{2,s},
			\end{array}$$
i.e., $h_s(b_{1, s})=b_{2, s}$, from which the equivalence
			$x_2 E_{\gA} x_1$ follows. This completes the proof of $(*)_{5.11.2}$. ]
			
Let us combine $(*)_{5.11.1}$ along with $(*)_{5.11.2}$ and deduce:
		\[
		\gB'_{\gA}/ E_{\gA} = Y_{s,\phi}/ E_{\gA} \simeq \gB_s.
		\]
	Thus, Claim \ref{cla6} follows.
	\end{proof}

	\iffalse
	
	$$\xymatrix{
		& 	 \gB\ar[rr] &&\gB	\\
		&&&\gB \gB&&\ar[ur]_{\widetilde{\Pi}_s}&\ar[ul]^{\widetilde{\Pi}_t}\\ 	
		&\gA\ar[uu]^{\subseteq} &&\ar[r]\gA\ar[uu]_{\subseteq}\\
		&\gA_t\ar[u]^{\pi_{2, t}}   &\gA\ar[l]^{\pi^{-1}_{1, t}}\ar[ul]^{\Pi_t}&\ar[uu]_{\subseteq}\gA\ar[uu]_{\subseteq.}\ar[uu]_{\subseteq}\ar[r]_{\pi^{-1}_{1, s}}\ar[ur]_{\Pi_s}  &\gA_s\ar[u]_{\pi_{2, s}}
		&&&&}$$
	\fi

	Now, we proceed the proof of  Theorem \ref{d2}.
	Given any
	$\gA \in \mathcal K_{\bfx, 1}$, we are going to define an isomorphism
	$$F_{\gA}: \gA \rightarrow \sort_1(\gB'_{\gA}/E_{\gA}),$$ uniformly definable from $\bfp$. To this end,
	we set
	\[
	F_{\gA}(a) :=\bigg\{\big [(\bar\pi, \bar g, \bar b)\big]_{E_{\gA}}: (\bar\pi, \bar g, \bar b)\in X_{\bfy_{\bfx}}(\gA) \text{~and~} \bigwedge_{s\in S} \pi_s(b_s)=a   \bigg \}.
	\]
In other words, $$F_{\gA}(a)=  \bigg\{ \big [(\bar\pi, \bar g, \bar b)\big]_{E_{\gA}}: (\bar\pi, \bar g, \bar b)\in \bold k(a)       \bigg  \}.$$
Due to Claim \ref{cla5}, we know that $\bold k(a)$ is an $E_{\gA}$-equivalence class, and in particular, $F_{\gA}(a)$ is a singleton. Recall from Claim \ref{cla6} that we have constructed the 2-sorted model $\gB'_{\gA} / E_{\gA}$. Now, for any $\gA \in \mathcal K_{\bfx, 1}$, we define $\gB_{\gA}$ as follows
	\begin{itemize}
		\item $\gB_{\gA}$ has as universe $|\gA| \cup \sort_2(\gB'_{\gA} / E_{\gA})$,
		\item the mapping $F_{\gA} \cup \id_{\sort_2(\gB'_{\gA}/E_{\gA})}$ is an isomorphism from $\gB_{\gA}$ onto $\gB'_{\gA}/E_{\gA}$.
	\end{itemize}
In summary, the function $F(\gA) = \gB_{\gA}$ is the required function we sought, and thus the theorem follows.
\end{proof}

\subsection*{Acknowledgements}
The second author's research has been supported by a grant from IPM (No. 1402030417). The third author is grateful to an individual who prefers to remain anonymous for providing typing services that were used during the work on the paper. The typists paid by Craig  Falls (2022-2025). The
third author would like to thank
ISF 1838/19: The Israel Science Foundation (ISF) (2019-2023).
ISF 2320/23: The Israel Science Foundation (ISF) (2023-2027)
NSF-BSF 2021: grant with M. Malliaris, NSF 2051825, BSF 3013005232 (2021-10-2026-09).
This is publication 1245 of third author. The authors sincerely thank the referee for their thorough review of the paper and for providing valuable comments.

\end{document}